\documentclass[aap]{amsart}

\newtheorem{theorem}{Theorem}[section]
\newtheorem{proposition}[theorem]{Proposition}
\newtheorem{lemma}[theorem]{Lemma}
\newtheorem{corollary}[theorem]{Corollary}
\theoremstyle{definition}
\newtheorem{definition}[theorem]{Definition}
\newtheorem{example}[theorem]{Example}

\theoremstyle{remark}
\newtheorem{remark}[theorem]{Remark}

\usepackage[utf8]{inputenc} 
\usepackage[T1]{fontenc}    
\usepackage{hyperref}       
\usepackage{url}            
\usepackage{booktabs}       
\usepackage{amsfonts}       
\usepackage{nicefrac}       
\usepackage{microtype}      
\usepackage{xcolor}
\usepackage{enumitem} 
\usepackage{amsmath,amssymb}
\usepackage{tikz-cd}

\DeclareFontFamily{OT1}{pzc}{}
\DeclareFontShape{OT1}{pzc}{m}{it}{<-> s * [1.200] pzcmi7t}{}
\DeclareMathAlphabet{\mathzapfc}{OT1}{pzc}{m}{it}
\usepackage[mathscr]{eucal} 

\renewcommand{\L}{\mathcal{L}} 
\renewcommand{\H}{\mathcal{H}}

\newcommand{\gen}{\mathcal{L}}

\newcommand{\A}{\mathzapfc{A}}
\newcommand{\bi}{\mathzapfc{B}}

\newcommand{\V}{\mathzapfc{V}} 
\renewcommand{\SS}{\mathzapfc{S}}
\renewcommand{\S}{\mathzapfc{S}}
\newcommand{\MM}{\mathzapfc{M}}  

\newcommand{\M}{\mathscr{M}}

\newcommand{\F}{\mathscr{F}} 
\newcommand{\G}{\mathscr G}

\newcommand{\X}{\mathfrak X}
\newcommand{\g}{\mathfrak g}

\newcommand{\R}{\mathbb{R}}

\newcommand{\refsec}[1]{§\ref{#1}}     
\newcommand{\Definition}[1]{\textbf{#1}} 
  
\newcommand{\hor}{\mathrm{hor}}
\newcommand{\ver}{\mathrm{ver}}

\newcommand{\metric}[2]{\left< #1, #2 \right>} 
\newcommand{\hvf}[1]{\setbox0=\hbox{$#1$}%
  \ifdim\wd0>1em\widehat{#1}\else\hat{#1}\fi} 
\def\rational#1#2{{\mathchoice{\textstyle{#1\over#2}}%
  {\scriptstyle{#1\over#2}}{\scriptscriptstyle{#1\over#2}}{#1/#2}}}
\def\half{\rational12}		

\newenvironment{talign*}
 {\csname align*\endcsname}
 {\endalign}

\def\dd{\mathbf{d}}

\newcommand{\dt}{\delta t}                 

\def\Or{\text{Or}}

\def\Dyn{\text{Dyn}}
\def\diff{\mathrm{d}}

\def\div{\text{div}}
\def\curl{\text{curl}}
\def\defn{\equiv}
\def\vol{\text{vol}}
\def\ad{\mathrm{ad}}

\def\tang{\partial}

\def\Tr{\mathrm{Tr}}
\def\Id{\mathrm{Id}}
\def\R{\mathbb R}

\begin{document}

\title{A Unifying and Canonical Description of Measure-Preserving Diffusions}
\author[Barp et al. ]{
Alessandro Barp\textsuperscript{1,2}, 
So Takao\textsuperscript{3}, Michael Betancourt\textsuperscript{4}, Alexis Arnaudon\textsuperscript{5}, Mark Girolami\textsuperscript{1,2}}
\address{\textsuperscript{1} Department of Engineering,
       University of Cambridge,
       Cambridge CB2 1PZ, United Kingdom}
\address{\textsuperscript{2} Alan Turing Institute, British Library, 96 Euston Rd, London NW1 2DB}

\address{\textsuperscript{3} UCL Centre for Artificial Intelligence, 
90 High Holborn, 
London WC1V 6LJ,     United Kingdom}
\address{\textsuperscript{5}  Blue Brain Project, École polytechnique fédérale de Lausanne (EPFL), Campus Biotech, 1202 Geneva, Switzerland}

\address{\textsuperscript{4} Symplectomorphic,  LLC,
        New York, USA}


\maketitle

\begin{abstract}
 A complete recipe of measure-preserving diffusions in Euclidean space was recently derived unifying several MCMC algorithms into a single framework. In this paper, we develop a geometric theory that improves and generalises this construction to any manifold.
 We thereby demonstrate that the completeness result is a direct consequence of the topology of the underlying manifold
and the geometry induced by the target measure $P$; there is no need to introduce other structures such as a Riemannian metric, local coordinates, or a reference measure.
Instead, our framework relies on the intrinsic geometry of $P$ and in particular its canonical derivative,  the deRham rotationnel, which allows us to parametrise the Fokker--Planck currents of measure-preserving diffusions using potentials. 
The geometric formalism can easily incorporate constraints and symmetries, and deliver new important insights, for example, a new complete recipe of Langevin-like diffusions that are suited to the construction of samplers. 
We also analyse the reversibility and dissipative properties of the diffusions, 
 the associated deterministic flow on the space of measures, and the geometry of Langevin processes.
Our article connects ideas from various literature and frames the theory of measure-preserving diffusions in its appropriate mathematical context.
\end{abstract}

\section{Introduction}

Markov processes play a prominent role in many areas of science.
In particular,
continuous {\em diffusion processes} that are designed to preserve a given target measure $P$ underpin numerous important algorithms.
For instance in physics, many models rely on stochastic Hamiltonian dynamics in which mechanical systems are coupled to a fluctuating thermostat process preserving the {\em Boltzmann--Gibbs distribution}, see~\cite{rapaport2004art, tuckerman2010statistical}.
These processes have also inspired various deep learning and optimisation methods \cite{cong2017deep, erdogdu2018global, gao2018global,
li2016high, leimkuhler2020constraint, shahbaba2019deep,xu2018global}, as well as sampling algorithms to approximate expectations of observables by generating a Markov chain $(X_i)$ composed of $P$-preserving transition kernels,
 $$\mathbb E_P[f] \defn \int f \dd P \approx \frac{1}{\ell}\sum_i^{\ell} f(X_i)\, , $$ see \cite{bou2014time, cances2007theoretical, heber2020posterior, leimkuhler2016computation,stoltz2010free}.
These include Hamiltonian Monte Carlo (HMC) samplers
which originated in lattice QCD \cite{clark2011improving,duane1987,takaishi2006testing},
and have, since then, been widely applied from chemistry to statistics
\cite{ahn2012bayesian,akhmatskaya2009comparison,Barp:2018,
betancourt2017conceptual,chen2014stochastic,Girolami2011,izaguirre2004shadow,
lelievre2016partial,
neal2011mcmc,welling2011bayesian,ma2019there}.
Measure-preserving diffusions are also used  to construct Stein operators via the generator approach \cite{Barbour2005,barbour1988stein}, with applications to inference, goodness-of-fit tests, measuring sample qualities and approximating distributions
\cite{barp2019minimum,chen2019stein,gorham2016measuring,gorham2015measuring,
Liu2016testing}.

A crucial prerequisite among these applications is the ability for practitioners to construct tailored measure-preserving diffusions; hence the need for a general characterisation and {\em recipe} to construct them.
In Euclidean space, such a recipe  was recently derived in \cite{ma2015complete}, extending previous results including \cite{hwang1993accelerating,rey2015irreversible,shi2012relation,
yin2006existence}. 
More precisely, they proved that  {\em any} continuous Markov process preserving a target measure of the form $P \propto e^{-H} \dd x$ and satisfying  an integrability assumption can be expressed as
\begin{equation}\label{eq:intro-complete}
\diff Z_t = - \left(Q\nabla H + D\nabla H \right) \diff t + \nabla \cdot \big(Q+D \big) \diff t + \sqrt{2D}\,\diff W_t\, ,
\end{equation}
where $Q$ and $D$ are antisymmetric and positive semi-definite position-dependent matrices respectively.

However, this recipe suffers from several important drawbacks preventing its use in some modern applications, in particular:
\begin{description}
\item[Theoretical]
From a theoretical viewpoint, the target distribution in \cite{ma2015complete} is specified with respect to the Lebesgue measure in Cartesian coordinates, and is thus inappropriate when other coordinates or reference measures are used.
Moreover, it  cannot easily incorporate geometric properties such as symmetries or conserved quantities that require a geometric framework. 
In addition,
the derivation in \cite{ma2015complete} relies on an ad hoc construction of $Q$ based on Fourier transforms, which cannot be generalised to arbitrary manifolds and requires an additional integrability assumption.
Instead, we wish to obtain a deeper understanding of the reason why $P$-preserving diffusions must take the form \eqref{eq:intro-complete}.
Ideally, the construction should solely depend on the target measure $P$, as this is the only object generally given in applications, and should properly incorporate the assumption that $P$ is a smooth distribution.
\item[Practical]
From the viewpoint of applications, the restriction of the recipe 
to Euclidean spaces dramatically restricts its scope, as many modern applications require 
$P$-preserving diffusions on manifolds.
These include thermodynamic integration, free energy calculation and molecular simulations
\cite{leimkuhler2016efficient,
lelievre2013optimal,lelievre2019hybrid,stoltz2010free},
spectral density estimation
of partially observed models for Bayesian methods, directional statistics,
 and lattice QCD calculations on compact Lie groups \cite{barp2019hamiltonian,brubaker2012family,Simon:2013,
 Diaconis:2013,graham2019manifold,Holbrook:2017,holbrook2016bayesian,
liu2016stochastic,Mardia:1999,zappa2018monte},
as well as applications  that are built using Stein operators on manifolds, diffusions on spaces with symmetries used in robotics (such as coupled rigid body motion), or for learning to encode symmetries in 
neural nets \cite{barp2018riemannian,chirikjian2010information,
ivancevic2010dynamics,
le2020diffusion,liu2018riemannian}.
\end{description}

\subsection*{Main contributions and Structure of the Paper}
In this article, we derive the   characterisation of $P$-preserving diffusions for any smooth positive measure $P$ on an arbitrary $n$-dimensional manifold $\M$, thus
 solving this problem in complete generality.
 Just as importantly, we derive  connections with other fields such as Poisson geometry, thermodynamics, and cohomology that explains how the construction fits in the bigger picture of mathematics.
 By leveraging the geometry generated by $P$, we provide a canonical construction for these diffusions, without introducing any other structures on $\M$, such as coordinates, a connection, a reference measure, a metric, or other mechanical structures like a Poisson bi-vector field.

Following the ideas from de Rham and Koszul,
we rely on the $P$-musical isomorphisms $P^{\flat}$ and $P^{\sharp}$ between multi-vector fields and twisted differential forms to construct the {\em $P$-rotationnel}
$\curl_P \defn P^{\sharp} \circ \dd \circ P^{\flat}$ (introduced in \refsec{sec:completeness-bracket-diffusion}), 
which transforms the  calculus of twisted differential forms associated to the exterior derivative $\dd$ into a measure-informed calculus of multi-vector fields 
associated to $\curl_P$.
The operators $P^{\flat} $ and $\curl_P$ induce the canonical geometry of the target $P$, which can then be used to show that {\em any} $P$-preserving diffusion on {\em any} manifold takes the form
\begin{align}\label{eq:intro-A-diff}
\begin{split}
\diff Z_t &= \underbrace{ \overbrace{\curl_P(\A) \,\diff t}^{\text{Fokker--Planck potential}}  + \overbrace{ P^{\sharp}(\gamma)\,\diff t}^{\text{topological obstruction}}  }_{\text{conservative } L^2(P)\text{-antisymmetric}} + \overbrace{\underbrace{\half  \div_{P}(Y_i)Y_i \,\diff t}_{\text{dissipative drift } } + \underbrace{ Y_i \circ \diff W^i_t}_{\text{Stratonovich noise}}}^{L^2(P)\text{-symmetric fluctuation-dissipation balance}}.
\end{split}
\end{align}
Here, the conservative term $\curl_P(\A) + P^{\sharp}(\gamma)$ characterises the set of all $P$-preserving vector fields.
Indeed, by analogy with the construction of the magnetic vector potential in physics,
we can locally build for any $P$-preserving vector field $X$, a `bi-vector potential' $\A$ such that $X=\curl_P(\A)$.
We will see that globally, this holds up to topological obstructions, which is represented by the term $ P^{\sharp}(\gamma)$, where $\gamma$ is a representative of the $(n-1)$-twisted de Rham cohomology group.
In addition to this, the dissipative term $\half  \div_{P}(Y_i)Y_i $ balances the fluctuations introduced by the arbitrary noise process $Y_i \circ \diff W^i_t$, as in the fluctuation-dissipation theorem of statistical physics, ensuring that the stochastic and dissipative components of \eqref{eq:intro-A-diff}, when combined, are also $P$-preserving.

Furthermore, when the target $P$ is expressed as $P = e^{-H} \mu_\M$, where $\mu_\M$ is an arbitrary smooth positive reference measure on $\M$, equation \eqref{eq:intro-A-diff} can be decomposed as follows:
\begin{equation}\label{eq:intro-A-diff-2}
\begin{split}
\diff Z_t = & \underbrace{ \overbrace{ \underbrace{ ( \quad  \curl_{\mu_\M}(\A)}_{\text{volume-preserving}}  +  \underbrace{ X^\A_H}_{\text{density-preserving}})\,\diff t}^{\text{local potential of Fokker--Planck current}}  \quad + \overbrace{ P^{\sharp}(\gamma)\,\diff t}^{\text{topological obstruction}}  }_{\text{conservative } L^2(P)\text{-antisymmetric drift}} \\
& + \overbrace{ ( \underbrace{-\half Y_i(H)Y_i }_{\text{density-dissipative drift } }+ \underbrace{ \underbrace{\half  \div_{\mu_\M}(Y_i)Y_i }_{\text{volume-dissipative drift } } )\,\diff t + \underbrace{ Y_i \circ \diff W^i_t}_{\text{Stratonovich noise}}}_{\text{volume-preserving } L^2(\mu_\M)\text{-symmetric diffusion}} }^{L^2(P)\text{-symmetric fluctuation-dissipation balance}},
\end{split}
\end{equation}
where $X^\A_H \defn \A(\dd H,\cdot)$ is a vector field generated by the log-density $H$, which we shall later refer to as an $\A$-{\em Hamiltonian vector field}.
In Euclidean space, we can show that \eqref{eq:intro-A-diff-2} boils down to  \eqref{eq:intro-complete}, thus showing that the integrability condition is not inherent to measure-preserving diffusions, but rather comes from the Fourier analysis used by the authors to derive the Euclidean recipe.

To derive this result, we begin in \refsec{sec:general-construction-diffusion} by building on ideas from Poisson mechanics
to construct a general $P$-preserving diffusion on $\M$ that naturally extends the Euclidean recipe \eqref{eq:intro-complete}.
 However, since the resulting diffusion is
 constructed from the Euclidean recipe,
 it does not take into account
 the non-trivial topological features  of the sample space, and is  thus
   incomplete, meaning that not every $P$-preserving diffusion can be expressed in that form. 
To remedy this,
we rely on the geometry induced by $P$  in \refsec{sec:completeness-bracket-diffusion} to derive the complete system \eqref{eq:intro-A-diff}.
In \refsec{sec:A-diffusion-compact}, we discuss the important case when $\M$ is compact, and show  how the topological obstructions can be directly expressed using harmonic forms of an arbitrary Riemannian metric.
The reversibility properties of \eqref{eq:intro-A-diff} will then be discussed in \refsec{sec:rever-dissip},
where we derive a general condition for the generator to be reversible up to some measure-preserving diffeomorphism, that generalises the momentum flip used in HMC to obtain a well-defined Metropolis--Hastings correction.
In \refsec{sec:volume-free-diffusions}, we 
give an example of some new insights gained by the geometric formulation,  by deriving a new recipe of {\em volume-free $e^{-H}\mu_\M$-preserving diffusions},  a subclass of \eqref{eq:intro-A-diff-2} in which the volume terms $\curl_{\mu_\M}(\A)$ and $\div_{\mu_\M}(Y_i)Y_i$ vanish.
The volume-free property is shared by both the underdamped and standard  overdamped Langevin processes, and it not only simplifies the practical implementation of the diffusion, but confers it important geometric guarantees, making  this volume-free subclass of diffusions particularly suited as a starting point to construct Langevin-like samplers. For the remaining sections, in \refsec{sec:examples-A-diff}, we study non-degenerate processes on manifolds, and characterise $P$-preserving diffusions that are expressed in terms of Riemannian It\^o noise.
In \refsec{sec:flow on space of measures}, we discuss the deterministic flow on the space of smooth measures associated to $P$-preserving diffusions, 
and derive a simple formula for the rate of change of functionals over measures, such as the KL divergence and other information entropies. 
In \refsec{sec:manifold SOL}, we discuss the geometry and generalisation of the underdamped Langevin diffusions to manifolds, which we can use to construct an irreversible Langevin-based sampler,
that generalises both the second-order Langevin HMC on $\R^n$, as well as the Euler--Poincaré diffusions on Lie groups, of which the original HMC algorithm  is a special case.
Finally, in 
 \refsec{ergodicity} and \refsec{sec:history-P-diffusions} we briefly discuss the ergodicity and history of measure-preserving diffusions.

\subsection*{Notation} 
Throughout, $\M$ is an arbitrary $n$-dimensional smooth manifold, and in particular, no assumptions are made on its connectedness, orientability, or compactness.
 We denote by 
$\mathfrak X^k(\M)$ and $\Omega^k(\M)$ the spaces of $k$-vector fields and $k$-forms, i.e., antisymmetric, contravariant and covariant tensor fields of rank $k$ respectively. 
The space of vector fields is denoted $\mathfrak X(\M) \defn \mathfrak X^1(\M)$, and given $X \in \mathfrak X(\M)$, $\L_X$ denotes the associated Lie derivative.
The exterior derivative on differential forms is denoted using the bold font $\dd $, and the space of smooth $\R$-valued functions on $\M$ is denoted $C^{\infty}(\M)$.
We say that $P$ is a {\em smooth measure} if it is  Radon measure that is absolutely continuous with respect to the null sets of the manifolds (generated by the charts).
This means that over  any coordinate chart $(x^i)$,  we can write $P = p \,\dd x$, where $\dd x$ is the local Lebesgue measure, and $p: \M \to \R$ is a measurable function, which we will assume to be smooth.
  We say that $P$ is positive if it is globally supported (its support is $\M$) and its local densities  $p$ above are positive smooth functions - in which case we denote its divergence on vector fields  by $\div_P:\mathfrak X(\M) \to C^{\infty}(\M) $, defined as $\div_P(X) P \defn \L_XP$,
  and let  $L^2(P)$ be the space of square $P$-integrable $\R$-valued functions on $\M$.
Given a diffeomorphism $\mathcal R: \M \to \M$, we denote by $\mathcal R^*$ and $\mathcal R_*$ the induced pullback and pushforward on tensor fields respectively.

\subsection*{Stochastic Differential Equations On Manifolds}

The standard It\^o stochastic processes, such as the Wiener and Ornstein-–Uhlenbeck processes, are typically defined over the real line.  Multivariate stochastic processes over the real numbers are then built up by injecting these processes along each coordinate axis. 
 A general manifold, however, is not rigid enough to admit this kind of construction without being endowed with a connection, which allows a multivariate stochastic process to be defined locally and then developed into a stochastic process that evolves over the manifold \cite{schwartz1982geometrie,malliavin1978geometrie,emery2012stochastic,norris1992complete,lazaro2007stochastic}. 
  Conveniently, Stratonovich processes do not require this machinery and can be defined over general manifolds using vector fields that direct the local noise  \cite{ikeda2014stochastic,lazaro2007stochastic}. We shall denote the Stratonovich differential by $\circ \, \diff$, and the standard Brownian motion on $\R^N$ by $W_t$~\cite{bismut1982mecanique,pavliotis2014stochastic}.

\section{A General Measure-Preserving Diffusion on Manifolds}
\label{sec:general-construction-diffusion}

In this section, we present a step-by-step construction of a general class of measure-preserving diffusions that extends the Euclidean recipe \eqref{eq:intro-complete} to manifolds in a natural way. To do this, our strategy is to inspect each term in \eqref{eq:intro-complete} and replace them with its natural counterparts on manifolds, and finally showing that the resulting diffusion is indeed measure-preserving.

First, we replace the Lebesgue measure $\dd x$ on $\R^n$ with an arbitrary smooth and positive reference measure
$\mu_\M$ on $\M$, hereafter referred to as the  \emph{volume measure} (note that there is no analogue of $\dd x$ on general non-homogeneous manifolds).
Typically, $\mu_\M$ is a  measure that is canonically induced by additional structures on $\M$, for instance a Haar, uniform, Riemannian or symplectic measure.
The assumptions on the target $P$ guarantee that it can be written as a {\em Gibbs distribution} with respect to some ``Hamiltonian'' function $H : \M \rightarrow \R$ and ``inverse temperature'' $\beta > 0$, given as
\begin{align}\label{gibbs-def}
P \defn p_\infty \,\mu_\M \propto e^{-\beta H} \mu_\M.
\end{align}
The Hamiltonian then represents the unnormalised log-density $H \propto \log p_\infty$ of the target measure.

Let us first consider the scenario in which the antisymmetric component $Q$ in  \eqref{eq:intro-complete} vanishes.
In this case, we observe that the complete recipe consists of a general noise contribution 
$\sqrt{2D}\,\diff W_t$ associated to an arbitrary positive semi-definite matrix $D$, corrected by a drift component 
$(- D \nabla H + \nabla \cdot D) \,\diff t$. This latter drift term is obtained by the following well-known result \cite[Prop. 4.5]{pavliotis2014stochastic} (the notion of reversibility will be discussed carefully in \refsec{sec:rever-dissip}).

\begin{proposition}[\cite{nelson1958adjoint}] Given a vector field $X$ on $\M = \R^n$, the target measure $P \defn p_\infty \dd x$ is a solution to the stationary Fokker--Planck equation for the  It\^{o} SDE $\diff Z_t= X(Z_t) \,\diff t+ \sqrt{2 D} \,\diff W_t$ when
$$ \nabla \cdot \mathfrak J( p_\infty) =0, \quad \text{where} \quad  \mathfrak J( p_\infty) \defn  p_\infty X - \nabla \cdot (p_\infty D).$$
Moreover, $\diff Z_t$ is a $p_\infty \dd x$-preserving reversible process, that is, the Fokker--Planck current  $\mathfrak J( p_\infty)$ vanishes, if and only if
$$\diff Z_t= \frac{1}{p_\infty} \nabla \cdot (p_\infty D)  \,\diff t+ \sqrt{2 D} \,\diff W_t = (- D \nabla H + \nabla \cdot D) \,\diff t + \sqrt{2 D} \,\diff W_t. $$
\end{proposition}
Hence, given the random noise $\sqrt{2D}\,\diff W_t$, the deterministic drift $(- D \nabla H + \nabla \cdot D) \diff t$ provides the necessary and sufficient correction to ensure that the diffusion is reversible with respect to the target measure $p_\infty \dd x$.
To extend this idea to manifolds, we begin by replacing the It\^{o} differential  in the noise contribution $\sqrt{2D} \,\diff W_t$, with Stratonovich differentials, as these do not require a connection on $\M$  \cite[Sec. 4]{norris1992complete}. 
Specifically, we replace the It\^{o} noise  $\sqrt{2D} \,\diff W_t$ with the Stratonovich noise $Y_i \circ \diff W^i_t$, where $\{Y_1, \ldots, Y_N\}$
is a generic family of `noise' vector fields.
To proceed further, we shall rely on the following useful result (proved in \ref{sec:proof-adjoint}).

\begin{lemma} Given smooth vector fields $X, Y_1, \ldots, Y_N$ on  $\M$, consider the Stratonovich SDE
\begin{equation}\label{general-SDE}
 \diff Z_t = X(Z_t) \,\diff t +  Y_i(Z_t) \circ \diff W^i_t\,.
\end{equation}
Its generator is given by
$
\gen f = Xf+\half  Y_iY_if
$, see for example \cite{bismut1981martingales}. The Fokker--Planck operator, viewed as the formal adjoint of $\L$ in $L^2(\mu_\M)$ is then
\begin{equation}  \label{Fokker-general}
  \gen^*f = \div_{\mu_\M}\left(-fX + \half Y_i(f)Y_i+ \half  f\div_{\mu_\M}(Y_i)Y_i \right) \, .
\end{equation} 
\end{lemma}

Now, for $P$ to be an invariant measure of the system, it suffices to show that $\gen^*e^{-\beta H} = 0$.
Mirroring the Euclidean case, we choose the drift $X$ to ensure that
 the {\em Fokker--Planck current}
 $$\mathfrak{J}(f) \defn -f X + \frac{1}{2} Y_i(f)Y_i+ \half  f \div_{\mu_\M}(Y_i)Y_i $$
  vanishes when $f$ is the target density $e^{-\beta H}$. This leads to the choice 
 $$ X \defn  - \half \beta Y_i(H)Y_i+ \half  \div_{\mu_\M}(Y_i)Y_i,$$
which naturally replaces the term $- D \nabla H + \nabla \cdot D$ in the Euclidean recipe.
Specifically, we have 
 $$\mathfrak J(e^{-\beta H})=0\quad \iff \quad \diff Z_t = \left( - \half \beta Y_i(H)Y_i+ \half  \div_{\mu_\M}(Y_i)Y_i\right) \diff t+ Y_i \circ \diff W^i_t, $$
 and in particular, the stationary Fokker--Planck equation
$$\L^*(e^{- \beta H})= \div_{\mu_\M}(\mathfrak J(e^{-\beta H}))  = 0$$
 is satisfied.
 We note that this
is an instance of the {\em fluctuation-dissipation relation},
 where the volume-distortion
 caused by the noise $Y_i \circ \diff W^i_t$ is exactly balanced 
by  the volume-dissipative drift $X \defn -\frac{\beta}{2}Y_i(H)Y_i \diff t+ \half \div_{\mu_\M}(Y_i)Y_i$.

We now consider the general case in which $Q$ does not necessarily vanish.
We note that geometrically, the drift term $Q\nabla H$ in \eqref{eq:intro-complete} represents a vector field that depends linearly on  the gradient of the log-density $H$.
While manifolds are generally non-linear spaces, their tangent and cotangent bundle provide linear spaces over each point $x \in\M$, namely, the tangent space $T_x\M$ of vectors, to which $Q \nabla H|_x$ belongs, and the cotangent space $T^*_x\M$ of covectors, to which $\nabla H|_x$ belongs.

Geometrically, maps from $\Omega^1(\M)$ to $\mathfrak X(\M)$, such as $\nabla H \mapsto Q \nabla H$ that are linear at each point on $\M$ are called {\em vector bundle morphisms}, since they are compatible with the vector bundle structures. 
A vector bundle morphism can  be conveniently  represented using a 
{\bf bracket} $\bi \in \Gamma(T\M \otimes T\M)$ (i.e., a contravariant tensor field of rank two),
which assigns to the log-density $H$, a {\bf $\bi$-Hamiltonian vector field} $X_H^{\bi}$ via the relation $X_H^\bi \defn \bi(\dd H,\cdot) $. Hence, we deduce that the linear operator $Q$ may be interpreted as a bracket $\bi$ on manifolds and the vector field $X_H^{\bi}$ is a natural candidate for generalising the term $Q\nabla H$ in the Euclidean recipe.


 \begin{example} $\bi$-Hamiltonian vector fields are ubiquitous in science. For example, they include Riemannian gradient flows, Hamiltonian vector fields associated to Poisson structures, and in  particular the  ones associated to a symplectic structure,
 4-gradient vector fields generated by a Lorentzian metric over spacetime, and thermodynamic flows,   which we shall come back to in \refsec{sec:rever-dissip}.
 \end{example}

 On the other hand, 
it is not a-priori clear how to interpret the final term ``$\nabla \cdot Q$''  in \eqref{eq:intro-complete} intrinsically,
as this would mean ``differentiate the rank two tensor field $\bi$ to turn it into a vector field'' -- such an operation does not exist  on general manifolds.
To proceed, we make the following ansatz for the measure-preserving diffusion on $\M$
\begin{equation}\label{Bracket-SDE}
 \diff Z_t = \big(X_{H }^\bi+Y\big) \diff t+\big( - \half \beta Y_i(H)Y_i+ \half  \div_{\mu_\M}(Y_i)Y_i\big) \diff t+ Y_i \circ \diff W^i_t,
\end{equation}
where $Y$ is currently an unspecified vector field that will later generalise the term ``$\nabla \cdot Q$'' in \eqref{eq:intro-complete}.
Meanwhile, we have the following result (proof in \refsec{sec:thrm-gibbs}):
\begin{theorem} \label{thrm-gibbs}
The Gibbs measure \eqref{gibbs-def} is preserved under the bracket diffusion \eqref{Bracket-SDE} 
if and only if the vector field $Y$ satisfies 
\begin{equation}
  \div_{\mu_\M}(X_{p_{\infty}}^\bi-\beta p_{\infty}Y)=0\, .
  \label{eq:X_def}
  \end{equation}
\end{theorem} 

We shall now use this result to 
make sense of the vector field $Y$.
Observe that by analogy with the  Euclidean case,  we want to obtain a complete recipe that is valid for any target distribution $P$, 
and thus for any choice of density $p_\infty$ with respect to a convenient volume measure $\mu_\M$. 
Since $Y$ should not depend on the choice of the target density, we expect the identity
\eqref{eq:X_def} to hold for {\em any} positive density $p_{\infty}$.
In particular, setting $p_\infty \equiv$ constant, we have $X_{p_\infty}^\bi=0$, implying that  $Y$ 
must preserve the reference measure, i.e., $\div_{\mu_\M}(Y)=0$.
If we assume this latter condition, then \eqref{eq:X_def} can be written as
\begin{equation}\label{eq:Weinstein-modular}
\div_{\mu_\M}(X_{p_{\infty}}^\bi)=\beta Y(p_{\infty}), 
\end{equation}
which holds  for any density $p_{\infty}$.
When $\bi$ is a {\em Poisson structure} (i.e., it is antisymmetric and satisfies the Jacobi identity), 
equation \eqref{eq:Weinstein-modular} is precisely the definition of \emph{modular vector fields} $Y$
in Poisson mechanics
(see \cite{dufour1991rotationnnels,dufour2006poisson,
weinstein1997modular}). 
Specifically, we can associate to a bracket $\bi$ and a volume measure $\mu_\M$, a {\bf modular vector field} $X^{\mu_\M}_\bi$, which is defined as the differential operator
\begin{equation}\label{eq:def-modular-field}
 X^{\mu_\M}_\bi:f \mapsto \div_{\mu_\M}(X_f^{\bi}),
\end{equation}
acting on  smooth functions.
While we have no reason to require that $\bi$ is Poisson, 
it is necessary for it to be antisymmetric in order for $X^{\mu_\M}_\bi$  to be a vector field, otherwise it will not satisfy the derivation property $X^{\mu_\M}_\bi(fg) = fX^{\mu_\M}_\bi(g) + gX^{\mu_\M}_\bi(f)$\footnote{To see this, check that $X^{\mu_\M}_\bi(fg) = \div_{\mu_\M}(fX_g^\bi + gX_f^\bi) = \bi(\dd g, \dd f) + f X^{\mu_\M}_\bi(g) + \bi(\dd f, \dd g) + g X^{\mu_\M}_\bi(f)$}.
Hereafter, we shall denote an arbitrary {\bf antisymmetric bracket} by $\A \in \mathfrak X^2(\M)$,
and note that the associated $\A$-Hamiltonian vector field $X_H^\A$ preserves the target density, i.e. $\L_{X_H^\A} p_\infty =0$, so the integral curves of $X_H^\A$ remain tangent to the surfaces of constant target density.
 Thus, replacing $Y$ with 
$X^{\mu_\M}_\A$ in \eqref{Bracket-SDE},
we obtain the following class of diffusions 
\begin{equation}\label{eq:modular-diffusion}
\diff Z_t = \overbrace{\underbrace{ X_{H }^\A  \,\diff t}_{e^{-H}\text{-preserving}}+\underbrace{ \beta^{-1}X^{\mu_\M}_\A  \,\diff t}_{\mu_\M\text{-preserving}} }^{e^{- \beta H}\mu_\M\text{-preserving}} - \overbrace{  \half \beta Y_i(H)Y_i  \,\diff t+ \underbrace{ \half  \div_{\mu_\M}(Y_i)Y_i \,\diff t+ Y_i \circ \diff W^i_t}_{\mu_\M\text{-preserving}}}^{e^{-\beta H} \mu_\M\text{-preserving}},
\end{equation}
which, by construction, preserves the target measure $P \propto e^{-\beta H} \mu_\M$.
Moreover, in the Euclidean case, \eqref{eq:modular-diffusion} recovers the Euclidean diffusion
$$
\diff Z_t = -Q\nabla H \diff t + \nabla \cdot Q  \diff t  - D\nabla H  \diff t +   \nabla \cdot D \diff t  + \sqrt{2D}\,\diff W_t\, ,
$$
 as shown in the following result (see proof in \refsec{sec:complete-recipe-recovery}).
\begin{corollary} Let $\M=\R^n$, $\mu_\M = \dd x$, $\sigma_{ij}\defn Y^i_j$, $D \defn\half \sigma \sigma^T$ and $Q_{ij} \defn \A^{ij}$. 
Then \eqref{eq:modular-diffusion} reduces to the It\^o diffusion \eqref{eq:intro-complete} derived in 
\cite{ma2015complete}.
\end{corollary}

\begin{remark}
We warn the readers about the similarity in the notations used for $\A$-Hamiltonian vector fields $X_H^\A$ and modular vector fields $X_\A^{\mu_\M}$. 
In the former, the bi-vector field $\A$ appears in the superscript, with a scalar function appearing in the subscript, while in the latter, the bi-vector field $\A$ appears in the subscript, with a positive measure appearing in the superscript.
\end{remark}

So far, we have explained how to give intrinsic meaning to each term of \eqref{eq:intro-complete} and
 constructed a class of diffusions that preserve the measure $P$. However, it is still unclear at this point whether this class of diffusions  \eqref{eq:modular-diffusion}, hereafter referred to as {\bf $\A$-diffusions}, is {\em complete}, that is, whether {\em any} $P$-preserving diffusion on $\M$ has the form \eqref{eq:modular-diffusion}.
In the following, we shall answer this question by taking into account the geometry of {\em volume manifolds}, given by the tuple $(\M,P)$, where $P$ is a smooth positive measure on $\M$. 

\section{Local and Global Completeness of $\A$-Diffusions} \label{sec:completeness-bracket-diffusion}

To work with the geometry and topology of $\M$ induced by our smooth positive target measure $P$, 
we rely on the $P$-{\em musical isomorphisms}, defined as follows.
Let  $\Omega^{k}_{\Or}(\M)$ be the space of $k$-twisted differential forms introduced by de Rham \cite{de1973varietes} (see for example \cite[Sec. 2.8]{frankel2011geometry} for an introduction to these important objects), while the unfamiliar reader may treat them as standard differential $k$-forms with loss of generality. 
The $P$-{\em flattening operator}
 $P^\flat : \mathfrak X^k(\M) \rightarrow \Omega^{n-k}_{\Or}(\M)$
for integers $0 \leq k \leq n$ is defined by  $P^\flat(X) \defn i_X P$, where $i_X$ denotes the interior product of (twisted) differential forms with a $k$-vector field $X$ \cite{koszul1985crochet,Marle:1997}.
More precisely, $P^{\flat}(X)$ is a twisted $(n-k)$-differential form such that  for any $(n-k)$-vector field $W$, we have 
 \begin{equation}\label{eq:P-flat}
 P^{\flat}(X)(W) =\metric{P^{\flat}(X)}{W}_* \defn  \metric{i_XP}{W}_*= \metric{P}{X \wedge W}_*,
\end{equation}   where $ \metric{\cdot}{\cdot}_*: \Omega^k(\M) \times \mathfrak X^k(\M) \to C^\infty(\M)$  denotes the duality pairing.
 Note that this map is well-defined regardless of whether $P$ is positive or not. When $P$  is globally supported, it becomes a $C^{\infty}(\M)$-linear isomorphism 
 $$P^{\flat}:\mathfrak X^k(\M) \xrightarrow{\sim} \Omega^{n-k}_\Or(\M),$$ with the inverse denoted by $P^\sharp$. We refer to the maps  $P^\flat$ and $P^\sharp$ as the {\em $P$-musical isomorphisms}, which should not be confused with the Riemannian musical isomorphisms.

\begin{example}  As an interesting example that shows the  relevance of the $P$-musical isomorphisms in the context of statistics, observe that if $Q$ is a smooth measure on $\M$,  then $P^\sharp$ yields the 
Radon--Nikodym derivative
$$ P^{\sharp}(Q) = \frac{\dd Q}{\dd P},$$
which can be observed directly from $Q = fP = P^\flat(f)$, where $f = \frac{\dd Q}{\dd P} \in C^\infty(\M)$.
\end{example}

Crucially for our purpose, the $P$-musical isomorphisms $P^\flat$ and $P^\sharp$ allow us to construct a canonical derivative associated to $P$, which will be central to our discussion on the completeness of $\A$-diffusions.
\begin{definition}[\cite{koszul1985crochet}]
The $P$-\Definition{rotationnel} of a $k$-vector field for some integer $1 \leq k \leq n$ is defined as 
$$ \curl_P \defn P^{\sharp} \circ \dd \circ P^{\flat}  : \mathfrak X^k(\M) \to \mathfrak X^{k-1}(\M), $$
where $\dd$ also denotes the extension of the exterior derivative to twisted forms \cite{bott2013differential}. When $k=0$, we set $\curl_P(f) \defn 0$ for any $f \in C^\infty(\M)$.
\end{definition}
An important  property of the operator $\curl_P$ is that it does not depend on the normalisation of $P$ (i.e., if $a \in \R-\{0\}$, then $ \curl_{aP} = \curl_P$), which is often unknown in applications. This is the case for example in Bayesian statistics (the posterior distribution), and in molecular dynamics (the canonical distribution).

In the context of measure-preserving diffusions, we are particularly interested in the action of $\curl_P$ on vector fields and bi-vector fields, which is given as follows.
In the case $k=1$ (i.e., vector fields), the $P$-rotationnel recovers the divergence operator
$$\curl_P|_{\mathfrak X(\M)} = \div_P,$$ 
and in the case $k=2$ (i.e., bi-vector fields),
 we have the following result showing that  $\curl_P$ maps bi-vector fields to their modular vector field  \eqref{eq:def-modular-field} (proved in  \refsec{sec:curl-is-modular}).
\begin{theorem}\label{lemma:div-curl}
Given a volume manifold $(\M,P)$, we have the identity 
$$\curl_{P}(\A) = X^{P}_\A.$$
\end{theorem}

It follows that we can replace the terms $X^{\mu_\M}_\A $ and $\div_{\mu_\M}(Y_i)Y_i $ in \eqref{eq:modular-diffusion} by $ \curl_{\mu_\M}(\A) $ and $\curl_{\mu_\M}(Y_i)Y_i $ respectively,
suggesting that the  rotationnel plays a central role in the construction of measure-preserving diffusions.
To realise the full potential of $\curl_P$ in our context, we need the following lemma (proved in \refsec{sec:fokker-density-adjoint}). 
\begin{lemma}  
The Fokker--Planck operator of the SDE \eqref{general-SDE}, viewed as the
 formal adjoint of the generator $\gen$ with respect to the dual pairing $\metric{f}{P}_* \mapsto \int f \dd P$
between smooth compactly supported functions $f$ and smooth measures $P$, is given by
\begin{equation}\label{eq:Fokker-adjoint-measures}
\L^*P = -\L_XP + \half \L_{Y_i} \L_{Y_i}P.
\end{equation} 
If $P$ is also positive, then the density of the smooth measure $\L^*P$ with respect to $P$ is given by
$\div_P \left( \half \div_P(Y_i)Y_i -X \right)$, and we call 
$$\mathfrak J(P) \defn \half \div_P(Y_i)Y_i -X$$ the {\bf Fokker-Planck current} of $P$.
Thus 
$$ \L^*P = \curl_P(\mathfrak J(P))P = \div_P(\mathfrak J(P))P.$$
\end{lemma}

Hence, the condition that $P$ is preserved by the diffusion,  namely $\L^*P=0$, reduces to the  condition that its Fokker--Planck current is $P$-preserving, i.e.,
\begin{equation}\label{eq:stationary-FP}
\L^*P=0 \quad \iff \quad \div_P \left( \mathfrak J(P)\right)=0.
\end{equation}
Now, as discussed in \refsec{sec:unpublished-results}, the $\R$-linearity of $\curl_P$, combined with the fact that it satisfies
$$\curl_P \circ \curl_P = P^{\sharp} \circ \dd \circ P^{\flat} \circ P^{\sharp} \circ \dd \circ P^{\flat} =
 P^{\sharp} \circ \dd \circ \dd \circ P^{\flat} = 0,$$
where we used that  $\dd \circ \dd = 0$, implies that $\curl_P$ is a \emph{boundary operator} on the space of multi-vector fields.
Accordingly, the measure $P$ defines \emph{homology groups}
$$ \mathcal H_P^{k} (\M) \defn \frac{\ker \left( \curl_P : \mathfrak X^{k}(\M) \to \mathfrak X^{k-1}(\M)\right)}{\text{Im} \left( \curl_P  : \mathfrak X^{k+1}(\M) \to  \mathfrak X^k(\M) \right)} $$
that inform us about the properties of $\curl_P$-free multi-vector fields on $(\M,P)$.
In particular, the first homology group $\H^1_P(\M)$ of $P$ describes the discrepancy between $\div_P$-free vector fields and  curl vector fields, i.e., vector fields of the form $\curl_P(\A)$ for some $\A \in \mathfrak X^2(\M)$,
and it follows that the space of $\div_P$-free vector fields, to which the Fokker--Planck current $\mathfrak J(P)$ of $P$-preserving diffusions belongs,
is isomorphic to 
$$ \ker \left( \div_P : \mathfrak X(\M) \to C^{\infty}(\M)\right) \cong  \text{Im} \left( \curl_P  : \X^2(\M) \to  \mathfrak X(\M) \right) \oplus \H^1_P(\M). $$
As a result, any $P$-preserving vector field $Z \in \mathfrak X(\M)$, i.e., $\div_P(Z)=0$,
can be expressed as a sum of (1) a globally curled component $\curl_P(\A)$ for some $\A \in \mathfrak X^2(\M)$, and (2) an additional term belonging to the first homology group $\H^1_P(\M)$, associated to topological obstructions.

It remains to characterise the elements of these groups, which can be achieved by noting that  the $P$-musical isomorphism $P^{\sharp}$ induces isomorphisms between the homology groups $\H^k_P(\M)$ and the twisted de Rham cohomology groups $\H_{dR}^{n-k}(\M)$ (see remark \ref{dR}), implying that
$$\H^1_P(\M) \cong P^{\sharp}(\H_{dR}^{n-1}(\M)).$$ Thus, the additional topological obstruction term may be parametrised by
$P^{\sharp}(\gamma)$, where $\gamma$ is a closed, twisted $(n-1)$-form whose de Rham class is non-zero. 
In conclusion, any such $Z$ can be expressed as 
$$Z= \curl_P(\A) + P^{\sharp}(\gamma),$$ with $\A \in \mathfrak X^2(\M)$ and $\gamma \in \H^{n-1}_{dR}(\M)$ (see theorem \ref{thrm:Dyn-is-loc-curl} for more details).

\begin{remark}[Twisted de Rham Cohomology]\label{dR} The twisted de Rham cohomology groups, defined as \cite{bott2013differential}
$$\H^k_{dR}(\M) \defn \frac{\ker \left( \dd : \Omega^k_{\Or}(\M) \to \Omega^{k+1}_\Or(\M)\right)}{\mathrm{Im} \left( \dd  : \Omega^{k-1}_\Or(\M) \to  \Omega^k_\Or(\M) \right)},$$
describe the topology of the sample manifold $\M$, such as the number of connected components of an orientable manifold $\M$, which is given by the $0$-th twisted de Rham cohomology group $\H^0_{dR}(\M)$.  
They also provide information on the solutions to the equation $\dd \alpha = \beta$ \cite[Chap. 18]{lee2013smooth}.
Similarly, the measure-informed homology groups $\H^{n-k}_P(\M) = P^\sharp(\H^{k}_{dR}(\M))$ inform us about the solutions of the equation $\curl_P(\mathcal V) = \mathcal W$, of which the stationary Fokker--Planck equation \eqref{eq:stationary-FP} is a special case. 
\end{remark}

\begin{remark}[Fokker--Planck Operators] The Fokker--Planck operator was defined as the adjoint of $\L$ with respect to the  pairing $\metric{f}{P}_* = \int_\M f \dd P$.
 However in the previous section, it was defined as the adjoint with respect to the pairing $C_c^{\infty}(\M) \times C^{\infty}(\M)\to \R$ defined as $\metric{f}{h}_{\mu_\M} = \int fh \dd \mu_\M$.
 The difference between these two pairings is that the former is a special case of the standard Poincaré bilinear form on the manifold $\M$, i.e., $\metric{\alpha}{\beta}_* \defn \int \alpha \wedge \beta$, while the latter  is associated to the  induced measure-informed bilinear form via the $P$-musical flattening, that is, 
 $\metric{\cdot}{\cdot}_P \defn \metric{\cdot}{P^{\flat} \cdot}_*$.
 The relation between the various definitions of the Fokker--Planck operator is then
 $$ \metric{\L f}{P}_* = \metric{f}{\L^*P}_*= \metric{f}{\div_P(\mathfrak{J}(P))P}_* = \metric{f}{\div_P(\mathfrak{J}(P))}_P = \metric{f}{\div_{\mu_\M}(\mathfrak J(e^{-H}))}_{\mu_\M}.$$
\end{remark} 

When focussing on just the {\em local} representations of $P$-preserving diffusions, we can also use the musical isomorphisms of $P$ to consider a {\bf ``$P$-twisted'' Poincar\'e lemma}\footnote{The standard Poincar\'e lemma states that any closed differential form $\alpha$ can be expressed locally as an exact form $\dd \beta$. The `$P$-twisted' Poincaré lemma is the  measure-informed analogue   to this on the space of multi-vector fields.}, which states that any vector field $Z$ that preserves the measure $P$ can be expressed locally as {\em curled vector fields} (see theorem \ref{thrm:Dyn-is-loc-curl}). More precisely, this means that we can find a neighbourhood $\iota _U:U \hookrightarrow \M$ around any point in $\M$, where $\iota_U$ is the inclusion, and a locally defined $\A \in \X^2(U)$ for which 
$$Z|_U= \curl_{\iota^*_U P}(\A)$$
holds.
Combining these results with the fact that the Fokker--Planck current of $P$-preserving diffusions is $\div_P$-free, we  obtain the following complete characterisation, or recipe, of $P$-preserving diffusions:

\begin{theorem}[Local and global completeness of $\A$-diffusions]\label{thm:completeness}
The smooth, positive target measure $P$
is a stationary measure of the general diffusion process
\begin{equation}\label{eq:arbitrary-diffusion}
\diff Z_t = X(Z_t) \diff t +  Y_i(Z_t) \circ \diff W^i_t\,
\end{equation}
 if and only if
the drift takes the form 
$X=\curl_{\iota_U^*P} (\A) + \half \div_{\iota_U^*P}(Y_i)Y_i$ on a neighbourhood $\iota_U:U \hookrightarrow \M$ of any point, for some local antisymmetric bracket $\A \in \mathfrak X^2(U)$. Thus, any such diffusions may be locally represented as 
\begin{equation}\label{eq:A-diff-theorem}
\diff Z_t  = \curl_{\iota_U^*P} (\A)\diff t + \half \div_{\iota^*_U P}(Y_i)Y_i \diff t +  Y_i(Z_t) \circ \diff W^i_t,
\end{equation}
which we will also refer to as $\A$-diffusions since as discussed below, they recover the original $\A$-diffusions \eqref{eq:modular-diffusion} when $P$ is written in the form \eqref{gibbs-def}.
Hence, the class of $\A$-diffusions is {\em locally complete}. Furthermore, they are {\em globally complete}, that is, any $P$-preserving diffusions is of the form 
\eqref{eq:A-diff-theorem} globally, if and only if the first homology group of $P$, or equivalently the $(n-1)^{st}$-twisted de Rham cohomology group, is trivial.
More generally, 
any $P$-preserving diffusion on any manifold can be expressed  as
$$ \diff Z_t  = \left(\curl_{P} (\A) + P^{\sharp}(\gamma) +\half \div_{P}(Y_i)Y_i \right)\diff t +  Y_i(Z_t) \circ \diff W^i_t,$$
for some $\A \in \mathfrak X^2(\M)$, and $\gamma \in H^{n-1}_{dR}(\M)$.
\end{theorem}

We stress that this result not only generalises the complete recipe derived in \cite{ma2015complete}, but also provides a re-interpretation of its derivation in terms of the canonical  geometry of the target measure $P$. A parallel can be made with the construction of potentials in physics.
Indeed, conservative fields in physics are usually represented by potentials. 
For example in classical mechanics, the Newton force $F$ is {\em conservative} if the net work done by it along any two (piecewise smooth) paths  $\lambda, \tilde{\lambda}$ with the same end points is the same, i.e., $ \int_{\lambda} F = \int_{\tilde{\lambda}} F$,  \cite[theorem 11.42]{lee2013smooth}.
Such force fields are represented by potential energy functions $V$, so we can find $V$ such that 
$F = \dd V$.
Similarly, in the theory of electromagnetism,  Gauss's and  Faraday's laws may be represented in the form $ \dd \bf{F}=0$, where $\bf{F}$ is the electromagnetic 2-form. These laws are equivalent to the existence of a magnetic potential 1-form $\bf{A}$ for which $\bf{F} = \dd {\bf}\bf A$~\cite{baez1994gauge}.

The situation is analogous in the context of $P$-preserving diffusions, except that we  need to adjust to the geometry induced by the target $P$, and use its canonical differential operator $\curl_P$ in place of the exterior derivative $\dd$.
As mentioned above, the condition that the diffusion \eqref{eq:arbitrary-diffusion} preserves $P$ is simply a condition that the Fokker--Planck current
$\mathfrak J(P) \defn \half \div_P(Y_i)Y_i -X$ conserves $P$, i.e., $\div_P(\mathfrak J(P))=0$.
Such conservative currents are locally represented by a `potential' $\A$, i.e., $\mathfrak J(P) = \curl_P(\A)$.
Surprisingly, this procedure is entirely canonical: it only depends on the volume manifold $(\M,P)$, which are the only objects we are given a-priori in many applications.
In particular on $\M=\R^n$, the de twisted de Rham cohomology is trivial, so this procedure holds globally. 
Thus we recover the result that if $Z$ is divergence-free (with respect to the Lebesgue measure),
i.e., $\nabla \cdot Z=0$, then there exists an antisymmetric matrix $\A$ for which $Z = \nabla \cdot \A$ \cite{mclachlan2002splitting}.
 This is all that is needed to obtain the Euclidean complete recipe of \cite{ma2015complete}.
 Consequently, our result shows that the integrability assumption in \cite{ma2015complete} is not needed - it is merely a consequence of the specific choice of potential constructed in the proof using Fourier transforms (see for example \cite{mclachlan2002splitting} for an alternative construction).
Indeed, we note that just as the potential energy of conservative forces has a gauge invariance defined by adding a constant to $V$, $V\mapsto V+c$ (which is the reason why HMC does not require knowledge of the normalising constant of the target measure), and $\bf{F}$ is invariant under the gauge transformation $\bf{A} \mapsto \bf{A} + \dd \psi$ $\,\,\forall \psi \in C^{\infty}(\M)$, the choice of antisymmetric bracket $\A$ has a natural gauge freedom obtained by
shifting $\A \mapsto \A + \curl_P(\mathcal V)$, with $\mathcal V \in \mathfrak X^3(\M)$ an arbitrary 3-vector field,
since
$$ \curl_{P}\left(\A + \curl_P(\mathcal V) \right) =  \curl_{P}\left(\A \right)+0=\curl_{P}\left(\A \right), $$
 thus making the choice of $\A$ non-unique. In \refsec{sec:volume-free-diffusions} we will discuss the benefits of choosing a ``conservative'' bracket, that is, $\A \defn \curl_{\mu_\M}(\mathcal V)$ for some $\mathcal V \in \mathfrak{X}^3(\M)$.

Finally, to recover the $\A$-diffusion  \eqref{eq:modular-diffusion} expressed in terms of a reference measure $\mu_\M$ from the canonical $\A$-diffusion \eqref{eq:A-diff-theorem}, 
we simply need to express the $\curl$ of $P=e^{-H}\mu_\M$ in terms of the curl of $\mu_\M$,
which may be achieved by noting that 
$$\curl_P \defn P^{\sharp} \circ \dd \circ P^{\flat} = \mu_\M^{\sharp} \circ \dd_H \circ \mu_\M^{\flat},$$ where $\dd_H \defn \dd - \dd H \wedge \cdot$ is the `distorted' exterior derivative \cite{guedira1984geometrie,rohm1986antisymmetric}.
This is precisely the expression for $\curl_{\mu_\M}$ with $\dd $ replaced by 
$\dd_H$ and one can check that the additional `twist' $\dd H \wedge \cdot$ generates both log-density terms $X_H^\A$ and $Y_i(H)Y_i$ in  \eqref{eq:modular-diffusion} (see \refsec{sec:derivation-of-canonical-A-diffusion} for a more detailed derivation).
Hence, if $P=e^{-H}\mu_\M$, we have
$$ \curl_P(\A) = \curl_{\mu_\M}(\A) + X^\A_H =X^{\mu_\M}_{\A} + X^\A_H, \quad \div_{P}(Y_i)Y_i = \div_{\mu_\M}(Y_i)Y_i - Y_i(H)Y_i.$$
It is interesting to observe that 
the $P$-preserving vector field 
$\curl_P(\A)$ always splits into a volume-preserving term $X^{\mu_\M}_\A$, i.e.,  $\L_{X^{\mu_\M}_\A}\mu_\M =0$, and a density-preserving term $X_H^\A$, i.e., 
$\L_{X_H^\A} p_\infty =0.$

\begin{remark}
Let us briefly explain why $\curl_{\mu_\M}$ is called ``curl''.
Note that when $\M = \R^3$ and $\mu_\M  = \dd \bf{x}$ $\defn \dd x \, \dd y \, \dd z$, we can write any bi-vector field $\A \in \mathfrak X^2(\R^3)$ as $\A  \defn \A _x \partial_  y \wedge \partial_  z + \A _y \partial_ z \wedge \partial_  x + \A _z \partial_  x \wedge \partial_  y$, and $\curl_{\mu_\M}$ corresponds to the classical curl of the ``vector  field'' $(\A_x,\A_y,\A_y)$:
$$ \curl_{\diff \bf{x}}(\A ) = \left( \frac{ \partial \A _y}{\partial z}-\frac{\partial \A _z}{\partial y} \right) \partial_x+
\left( \frac{ \partial \A _z}{\partial x}-\frac{\partial \A _x}{\partial z} \right) \partial_y+
\left(\frac{ \partial \A _x}{\partial y}- \frac{\partial \A _y}{\partial x} \right) \partial_z. $$
In particular, any $\dd \bf{x}$-preserving vector field can be written as above, which we may also view as a sum of Hamiltonian vector fields on the coordinate 2-surfaces:
$$
 \curl_{\diff \bf{x}}(\A ) = \underbrace{\left( \frac{ \partial \A _y}{\partial z}\partial_x- \frac{\partial \A_y }{\partial x} \partial_z \right)}_{\text{Ham. field of } \A_y \text{ on x-z plane}}+ 
\underbrace{\left( \frac{ \partial \A_z }{\partial x}\partial_y
- \frac{\partial \A_z }{\partial y} \partial_x  \right)}_{\text{Ham. field of } \A_z \text{ on x-y plane}} +
\underbrace{\left(\frac{ \partial \A _x}{\partial y} \partial_z -\frac{\partial \A _x} {\partial z} \partial_y\right)}_{\text{Ham. field of } \A_x \text{ on z-y plane}} .$$
(Recall the Hamiltonian vector field of $H$ in Darboux coordinates $(q,p)$ is
$ \frac{ \partial H}{\partial p}\partial_q- \frac{\partial H }{\partial q} \partial_p$). 
\end{remark}

\begin{remark}
In the Euclidean space recipe, the terms associated with the $2^{nd}$-order tensors
$ \nabla \cdot D$  and $\nabla \cdot Q$ look identical -- they are obtained by applying the ``divergence'' $\nabla \cdot$ to $D$ and $Q$, which are both second-order objects (i.e., tensors with two indices).
In our geometric formulation, we observe that $\nabla \cdot D$ is actually obtained by differentiating  the noise vector fields $\div_P(Y_i)Y_i$ which are first-order tensors,  while $\nabla \cdot Q$, corresponding to the term $\curl_{P}(\A)$ in our formulation, genuinely involves differentiating a second-order tensor $\A$. 
In particular, while it is true that $\curl \circ \curl=0$, the equality $\nabla \cdot \nabla \cdot =0$ is only valid when applied to the antisymmetric component, i.e.,
$$ \nabla \cdot \nabla \cdot D \neq 0, \qquad \nabla \cdot \nabla \cdot Q=0.$$
The geometric formulation properly distinguishes these operations:
$$ \div_{\mu_\M}\left( \div_{\mu_\M}(Y_i)Y_i \right) = Y_i(\div_{\mu_\M}(Y_i))+ \left(\div_{\mu_\M}(Y_i)\right)^2,\qquad \div_{\mu_\M} \left(\curl_{\mu_\M}(\A) \right)=0, $$
so that while $\curl_{\mu_\M}(\A)$ is always volume-preserving, this is generally not the case for $\div_{\mu_\M}(Y_i)Y_i$ (an important exception is the Langevin diffusion as we shall discuss in \refsec{sec:volume-free-diffusions}).
\end{remark}

In the next section we will describe the important scenario where the sample space $\M$ is compact, in which case the topological obstructions can be represented explicitly in terms of harmonic forms.

\section{Measure-Preserving Diffusions on Compact Manifolds}\label{sec:A-diffusion-compact}

We saw previously 
that when the topology  of the sample space contain topological obstructions, it is necessary to add an additional term in $\A$-diffusions representing the non-triviality of the homology of $P$.
For compact orientable manifolds, we may use the de Rham--Hodge--Kodaira's decomposition of differential forms to construct the Fokker--Planck current of general $P$-preserving measures, as shown in the following theorem (proved in \refsec{sec:Riemannian-proof}).
\begin{theorem}\label{thm:diffusion-compact}
 Let $\M$ be a compact orientable Riemannian manifold and let $\mu_\M \defn \vol$ and $\nabla \cdot$ denote respectively the Riemannian measure and divergence.
Then, any $e^{-H}\vol$-preserving diffusion has the form
 \begin{equation}\label{eq:diffusion-compact}
\diff  Z_t \defn  \left(X_H  + \half \left(    \nabla \cdot Y_i -Y_i(H) \right)Y_i -   \nabla \cdot \A  +  e^{H} \sharp \star ^{-1}\zeta  \right) \diff t + Y_i \circ \diff W^i_t,
\end{equation}
where $\star$ is the Hodge star operator, $\sharp$ is the Riemannian musical isomorphism, $\A$ is an antisymmetric tensor and $\zeta$ is a harmonic $(n-1)$-form (i.e., it satisfies ``Maxwell's equations'' $\diff \zeta=0$, $\diff \star \zeta =0$).
\end{theorem}

A similar result can be found in \cite{ikeda2014stochastic}, although the authors assume that the diffusion is also non-degenerate (i.e., its generator is elliptic) in order to obtain a Riemannian metric from the noise process (the Riemannian Brownian motion), which is then used to turn the Fokker--Planck current $\mathfrak J(P)$ into a 1-form that can be analysed through its de Rham--Hodge--Kodaira decomposition.
In order to clarify the roles played by 
 the assumptions of compactness and non-degeneracy, we have treated them separately, the latter which can be found in \refsec{sec:examples-A-diff}.

Note that in the above theorem \ref{thm:diffusion-compact}, we do not make any assumptions on the noise, although it assumes that we can express our target measure $P$ in terms of the Riemannian measure $\vol$, which might be inconvenient in practice. 
Topologically, the  presence of the harmonic term in \eqref{eq:diffusion-compact} may be understood from the fact that on compact orientable manifolds, the twisted de Rham cohomology groups are isomorphic to the space of harmonic forms. 
It follows that  $P$-preserving vector fields have the form $\curl_P(\A)+ P^{\sharp}(\zeta)$ for an antisymmetric bracket $\A$ and a non-zero harmonic $(n-1)$-form $\zeta$ (see \refsec{sec:unpublished-results}).
 Hence, on compact manifolds, any $P$-preserving diffusions take the form
$$\diff Z_t = \underbrace{ \curl_{P}( \A) \, \diff t +  \half \div_{P}(Y_i)Y_i \,\diff t+Y_i \circ \diff W^i_t}_{\A\text{-diffusion}}+ \underbrace{ P^{\sharp}(\zeta)\,\diff t}_{\text{harmonic obstruction}} ,$$
for some harmonic $(n-1)$-form $\zeta$ associated to an arbitrary Riemannian metric.

\section{Reversibility}
\label{sec:rever-dissip}
Recall that the Euclidean recipe  for measure-preserving diffusions \eqref{eq:intro-complete} depends entirely on an antisymmetric matrix $Q$ and a symmetric positive semi-definite matrix $D$. On the other hand, the geometric generalisation we have derived in previous sections is constructed using a bi-vector field $\A$ and a set of noise-vector fields $(Y_i)$.
 In order to make this connection clearer, 
we now discuss the symmetric/antisymmetric decomposition of $\A$-diffusions and its relation to the notion of reversibility. 
For this, we first note that the noise vector fields $\{Y_i\}_{i=1}^N$ canonically generate a symmetric bracket, denoted 
$\SS \defn Y_i \otimes Y_i$, by setting
$$\SS(\dd f,\dd g)  \defn Y_i(f)Y_i(g) \quad \text{ for any } f,g \in C^{\infty}(\M)\, .$$
Introducing the notation
\begin{equation}\label{eq:bracket notation}
\{f,g \}_\bi\defn \bi(\dd f, \dd g),
\end{equation} 
for a general bracket $\bi$,
we see that the symmetric bracket defined above is \emph{dissipative}, in the sense that it satisfies the dissipative property
$$  \{ f,f\}_\SS = \sum_i Y_i(f)^2 \geq 0\, ,$$
which further implies 
$$ \{ f,g\}_\SS^2 \leq  \{ f,f\}_\SS  \{ g,g\}_\SS.$$
This contrasts with the antisymmetric bracket, which has the conservative property
$$ \{ f, f\}_\A = 0\, .$$
\begin{remark}
In the context of mechanics, special forms of (symmetric) dissipative  brackets $\SS$ have been considered by several authors to model dissipative components of mechanical systems in an attempt to cast these systems from an algebraic framework. Examples include the metriplectic bracket \cite{kaufman1984dissipative,
morrison1986paradigm,grmela1986bracket,
guha2007metriplectic,materassi2016entropy}, double-bracket \cite{brockett1991dynamical,bloch1996euler}, and selective-decay bracket \cite{gay2013selective}. Whereas in these works the symmetric structures are constructed in an ad hoc manner, it would be interesting to understand them as arising from noise vector fields chosen to model fluctuations, as we do here. This connection shall be further explored in a separate paper by the authors.
\end{remark}

\begin{example}
When the bi-vector field $\A$ has nice properties such as symmetries, it can be desirable to employ it to construct the noise vector-fields. This may be done using ``noise functions" $(H_i)$ and choosing $Y_i \defn X^\A_{H_i}$. Such a mechanism was used to construct a coordinate-independent irreversible MCMC sampler on Lie groups in  \cite{arnaudon2019irreversible}. In that case, the bracket $\SS$ safisfies 
$\SS(\dd f, \dd g)= \A( \dd H_i,\dd f) \A( \dd H_i, \dd g)$, and the generator  of the diffusion has a double bracket form (this should not be confused with the notion of ``double brackets'' in the sense of Brockett and Bloch).
\end{example}
A bracket that is decomposed into the sum of an antisymmetric and dissipative bracket is known as a \emph{thermodynamic bracket}. It follows that  $P$-preserving diffusions are parametrised by thermodynamic brackets up to topological obstructions.
Thus, many properties of the diffusion can be studied through its thermodynamic bracket; for example  in \refsec{sec:flow on space of measures}, we will see that the thermodynamic  bracket of the diffusion provides a simple formula to evaluate the rate of change of functionals on volume measures along the diffusion.

If we define the divergence of the dissipative bracket $\SS$ constructed above  by $\div_P(\SS)\defn  \div_P(Y_i)Y_i$, we can decompose the drift of the $\A$-diffusion into components associated with the dissipative and antisymmetric brackets:
\begin{equation}\label{eq:dissipative-form-A-diffusion}
\diff Z_t = \underbrace{\curl_{P}( \A) \, \diff t}_{\text{antisymmetric}} + \underbrace{\half \div_{P}(\SS) \,\diff t}_{\text{dissipative}}+\underbrace{Y_i \circ \diff W^i_t}_{\text{noise}},
\end{equation}
which for $P = e^{-\beta H} \mu_\M$, further decomposes into (compare with \eqref{eq:intro-complete})
\begin{equation}\label{eq:dissipative-form-A-diffusion-2}
 \diff Z_t = \beta\left(X_H^\A - \half X_H^\SS \right) \diff t+ \left(  \half \div_{\mu_\M}(\SS)+\curl_{\mu_\M}(\A) \right) \diff t+ Y_i \circ \diff W^i_t.
\end{equation}

The following corollary derives a decomposition of the corresponding generator into symmetric and anti-symmetric parts, generalizing the standard result in the Euclidean case \cite{pavliotis2014stochastic}. This result is useful as for instance, it enables us to build the most general Stein operators on manifolds available from the generator approach.
\begin{corollary} The generator of a $P$-preserving diffusion expressed in the form of \eqref{eq:intro-A-diff} can be written as (recall the definition of the differential operator $X^P_\bi$ in \eqref{eq:def-modular-field})
\begin{equation}\label{eq:generator-sym-anti}
\L f = \underbrace{ X^{P}_\A(f) +P^{\sharp}(\gamma)(f)}_{L^2(P)\text{-antisymmetric}} +  \underbrace{\half X^{P}_\SS(f)}_{L^2(P)\text{-symmetric}}.
\end{equation}
Moreover, $\half X^{P}_\SS$ is  symmetric in $L^2(P)$, while $X^{P}_\A$ and $P^{\sharp}(\gamma)$  are both  antisymmetric  in $L^2(P)$. That is,
$$ \metric{X^P_\A f}{h}_P= -\metric{f}{X^P_\A h}_P , \qquad \metric{X^P_\SS f}{h}_P= \metric{f}{X^P_\SS h}_P,  $$
where $\metric{\cdot}{\cdot}_P$ denotes the $L^2(P)$ pseudo-inner product, $\metric{f}{h}_P \defn \int fh \,\dd P$.
Hence, the generator $\L$ is symmetric  if and only if  $X_\A^P +P^{\sharp}(\gamma)=0$.
In general, the generator of \eqref{general-SDE}  satisfies  $\L = \half X^P_\SS$ if and only if the Fokker-Planck current of $P$  vanishes,   in which case, we say that $\L$ satisfies the \Definition{detailed balance condition}, and the diffusion is \Definition{reversible}.
Finally, we have that $\half X^{P}_\SS$ is non-positive, i.e., 
$$\metric{\half X^{P}_\SS(f)}{f}_P \leq 0$$
for all $f \in C_c^{\infty}(\M)$.
\end{corollary}
\noindent We refer the readers to \refsec{sec:proofs-reversibility-sec} for the proof.

\begin{remark}[Carré du champ operator]
The generator $\L$ of any diffusion defines a \emph{carré du champ operator} by
$$ \Gamma(f,h) \defn \half \left( \L(fh)-f \L h - h \L f\right) $$
 over appropriate algebras of functions. These  play an important role  in the study of reversible diffusions (see
 \cite{bakry2013analysis} and references therein). Using \eqref{eq:generator-sym-anti}, we see that the symmetric bracket $\S$ is in fact equivalent to the carré du champ operator $\Gamma$.

\begin{corollary}
For any $f,h \in C^{\infty}(\M)$, the carré du champ operator of a $P$-preserving diffusion is precisely the dissipative bracket generated by the noise
$$ \Gamma(f,h)= \{f,h\}_\S.$$
\end{corollary}
\end{remark}

The $L^2(P)$-symmetry of reversible Markov processes confers them important theoretical properties that are useful for example in the study of their convergence to equilibrium 
\cite{bakry2013analysis,
pavliotis2014stochastic}.
However, they also form a  restrictive class of diffusions that often have slow convergence properties \cite{duncan2017nonreversible,ottobre2016markov}.
The decomposition of the generator above allows us to show that for appropriate transformations $\mathcal{R} : \M \to \M$, the generator of the $\A$-diffusion is \emph{reversible up to $\mathcal{R}$},
which extends the notion of ``reversibility up to momentum flip'' of the Langevin diffusion in Euclidean space, usually associated with improved mixing properties~\cite{fang2014compressible,sohl2014hamiltonian,stoltz2010free}.
\begin{corollary} Let $\mathcal{R}$ be a target-preserving diffeomorphism,  which is an $\A$-antimorphism  and a $\SS$-morphism, that is
$$\mathcal{R}_*\A=-\A , \qquad \mathcal{R}_*\SS= \SS.$$
 Then, the generator of the $\A$-diffusion \eqref{eq:generator-sym-anti} is \Definition{reversible up to $\mathcal{R}$}. That is, we have
$$ \metric{f}{\L h}_P = \metric{\L \mathcal{R}^*f}{\mathcal{R}^*h}_P, \quad \forall f,h \in C_c^{\infty}(\M).$$
\end{corollary}

\begin{example} 
For instance, if $\mathcal R_* Y_i = \pm Y_i$, then $ \mathcal{R}_*\SS= \SS$, and this is precisely what happens in the underdamped Langevin diffusion on phase space $(q,p)$ (see below), wherein $Y_i$ are proportional to $\partial_p$, and the momentum-flip $\mathcal R:(q,p) \mapsto (q,-p)$  flips the noise fields $\mathcal R_* Y_i = - Y_i$.
\end{example}

Combining the results from the previous sections,  we have the following interpretations of the various components of measure-preserving diffusions:
\begin{equation}\label{eq:A-diffusion-detailed}
\begin{split}
\diff Z_t = & \underbrace{ \overbrace{ \underbrace{ ( \quad X^{\mu_\M}_{\A}}_{\text{volume-preserving}} +  \underbrace{ X^\A_H}_{\text{density-preserving}})\,\diff t}^{\text{local potential of Fokker--Planck current}}  \quad + \overbrace{ P^{\sharp}(\gamma)\,\diff t}^{\text{topological obstruction}}  }_{\text{conservative } L^2(P)\text{-antisymmetric drift}} \\
& \qquad \qquad + \underbrace{ ( \underbrace{-\half Y_i(H)Y_i }_{\text{density-dissipative drift } }+ \overbrace{ \underbrace{\half  \div_{\mu_\M}(Y_i)Y_i }_{\text{volume-dissipative drift } } )\,\diff t + \underbrace{ Y_i \circ \diff W^i_t}_{\text{Stratonovich noise}}}^{\text{volume-preserving } L^2(\mu_\M)\text{-symmetric diffusion}} }_{L^2(P)\text{-symmetric fluctuation-dissipation balance}}.
\end{split}
\end{equation}
In particular, we note the following:
\begin{enumerate}[ label= \bf(\roman*)]
\item  $X_H^\A$ (or $ -Q\nabla H $) is the $p_\infty$-preserving ($X_H^\A(p_\infty)=0$) $\A$-Hamiltonian vector field generated by an antisymmetric bracket $\A$;
\item $X^{\mu_\M}_\A $ (or $\nabla \cdot Q$) is a generalisation of the modular vector field from Poisson mechanics, which preserves the volume measure $\mu_\M$ and describes how the $\A$-Hamiltonian vector fields $X_f^\A$ distort the reference measure $\mu_\M$, i.e., 
$X^{\mu_\M}_\A =0$ iff  $X_f^\A$  preserves $\mu_\M$ for all $f$. 
\item When $X^{\mu_\M}_\A$ is added to $X_H^\A$, the resulting vector field $X^P_\A$ is $P$-preserving.
On contractible manifolds such as $\R^n$, the Fokker--Planck current of any $P$-preserving diffusion can be written as $X^P_\A$ for some $\A$; otherwise, for (global) completeness, a topological obstruction term $P^{\sharp}(\gamma)$ parametrised by  the $(n-1)^{th}$-twisted de Rham cohomology group $\gamma \in H^{n-1}_{dR}(\M)$ must also be added by theorem \ref{thm:completeness}. 
The resulting generator $X^P_\A+P^{\sharp}(\gamma)$  is antisymmetric in $L^2(P)$. On compact oriented manifolds, this topological contribution can be parametrised by harmonic forms, as discussed in \refsec{sec:A-diffusion-compact}. 
\item $\div_{\mu_\M}(Y_i)Y_i$ (or $\nabla \cdot D$ minus the It\^o-to-Stratonovich correction) represents the distortion of the volume measure along the noise vector fields, and is usually dissipative; \label{vol-distortion-term}
\item  $-Y_i(H )Y_i$ (or $-D\nabla H $) is the rate of change of the target log-density  along the noise fields. \label{target-distortion-term}
\item The overall noise contribution \ref{vol-distortion-term} + \ref{target-distortion-term} is generated by the second-order ``modular'' operator 
$\half X^P_{\S}$, which is symmetric in $L^2(P)$.
\end{enumerate}


\section{Complete Recipe of Volume-Free $P$-preserving Diffusions}\label{sec:volume-free-diffusions}

As noted in the introduction, obtaining a complete recipe of $P$-preserving diffusions 
allows practitioners to focus 
on the tuning of the parameters $\{\A, (Y_i)\}$, as well as its numerical implementation. 
Using the geometric formalism, we now discuss a particularly interesting class of parameters  inspired by
two classes of $P$-preserving diffusions that play a particularly central role in many applications, namely, the underdamped  and overdamped Langevin processes, which are used for example to construct MALA and HMC respectively.
 
The overdamped Langevin process on $\R^n$ targets a measure of the from 
$P \propto e^{-H} \dd q$, 
and corresponds to the choices
$Q=0$ and $D =$ arbitrary positive-definite constant matrix in the Euclidean recipe \eqref{eq:intro-complete}:
\begin{equation}\label{eq:overdamped-langevin-Rn}
 \diff Z_t = -D\nabla H+ \sqrt{2D}\diff W_t.
\end{equation}
Since $D$ is a positive-definite contravariant tensor, we may think of it as a Riemannian co-metric. Thus, the drift 
$D\nabla H$  corresponds to a Riemannian gradient flow. On the other hand, the underdamped Langevin process evolves on the phase space $\R^n \times \R^n$ and preserve target measures of the form 
$P = \mu_H \propto e^{-H(q,p)} \dd q \,\dd p$.
Starting from the Euclidean complete recipe, this is obtained by setting
$$ Q \defn -J\defn -\begin{pmatrix} 0 & I \\ -I & 0 \end{pmatrix},
\qquad 
 D \defn 
\begin{pmatrix} 0 & 0 \\ 0 & C \end{pmatrix}  $$
where $J$ is the symplectic matrix and $C\in \R^{n\times n}$  is a positive semi-definite matrix \cite[Sec. 2.2.3]{stoltz2010free}. This gives us the  second-order Langevin process (also called the underamped Langevin process)
\begin{subequations} \label{Langevin-process-SDE}
\begin{align} 
 \dd Q_t &= \nabla_p H(Q_t,P_t) \dd t, \\
 \dd P_t &= -\nabla_q H(Q_t,P_t) \dd t -C \nabla_p H (Q_t,P_t)\dd t + \sqrt{2 C} \dd W_t, 
\end{align}
\end{subequations}
which models the fluctuating dynamics of a Hamiltonian system coupled to a thermostat. These systems have been widely used to construct samplers, as shown in \cite{cheng2017underdamped,dobson2019reversible,
leimkuhler2016efficient,ottobre2016markov,stoltz2010free}, by splitting it into a Hamiltonian and thermostat process.

 An interesting property that they both share is that, while being measure-preserving, they appear to be  fully specified by the target log-density $H$ and the random noise term.
In other words,  the reference measure terms ($\curl_{\mu_\M}(\A)$ and $\div_{\mu_\M}(Y_i)Y_i$) are non-existent in both cases, which greatly simplifies the implementation of such processes. Thus, in this section, we are interested in characterising the sub-class of $\A$-diffusions
\begin{equation}\label{eq:volume-split}
 \diff Z_t = \underbrace{\left(X_H^\A - \half \beta X_H^\SS \right) \diff t}_{\text{density-drift}}+ \underbrace{ \left(  \half\div_{\mu_\M}(\SS)+\curl_{\mu_\M}(\A) \right) \diff t}_{\text{reference-measure-drift}}+ \underbrace{ Y_i \circ \diff W^i_t}_{\text{random noise}},
\end{equation}
for which the 
terms involving the reference-measure $\mu_\M$ (we call this the `reference-measure drift' in \eqref{eq:volume-split}) vanish. We refer to this subclass of $\A$-diffusions as {\em volume-free $P$-preserving diffusions}.

While it is unclear how to even approach this problem from the Euclidean recipe/formalism, our geometric formalism provides an immediate characterisation of such processes in the case where the noise-fields $(Y_i)$ are (pointwise) linearly independent. 
Indeed, in this case,  we have
$$ \div_{\mu_\M}(Y_i)Y_i =0 \quad \iff \quad Y_i = \curl_{\mu_\M}(\A_i)+P^{\sharp}(\alpha)$$
and 
$$ \curl_{\mu_\M}(\A) =0 \quad \iff \quad \A = \curl_{\mu_\M}(\mathcal V) +P^{\sharp}(\beta),$$
for some 3-vector field $\mathcal V \in \mathfrak X^3(\M)$, noise bi-vector fields $\A_i \in \mathfrak X^2(\M)$ (these are unrelated to the deterministic bi-vector field $\A$), and  appropriate topological contributions $\alpha \in H^{n-1}_{dR}(\M), \beta \in H^{n-2}_{dR}(\M)$ (see theorem \ref{thrm:Dyn-is-loc-curl}).

Hence, ignoring topological obstructions, volume-free $P$-preserving diffusions are characterised by 
parameters obtained through the rotationnels of higher-order tensors,
$$ \{\A, (Y_i) \} = \{\curl_{\mu_\M} (\mathcal V), \left(\curl_{\mu_\M}(\A_i)\right) \}.$$
In particular, up to topological obstructions, such diffusions take the form  
\begin{equation}\label{eq:vol-free-diffusion}
\diff Z_t =  X_{H }^{\curl_{\mu_\M}(\V)}\diff t  -  \half X_{H}^{\curl_{\mu_\M}(\A_i) \otimes \curl_{\mu_\M}(\A_i)}  \diff t+ \curl_{\mu_\M}(\A_i) \circ \diff W^i_t
\end{equation}
and moreover, yield several geometric guarantees:
first of all, the noise fields $(Y_i)$ are automatically $\mu_\M$-preserving (i.e., $\div_{\mu_\M}(Y_i) = 0$) by construction, 
and secondly, any $\A$-Hamiltonian vector fields $X_f^\A$ are also $\mu_\M$-preserving for any $f \in C^{\infty}(\M)$, since by theorem \ref{lemma:div-curl} we have
$$0 = \curl_{\mu_\M}(\A)(f) = X_\A^{\mu_\M}(f) = \div_{\mu_\M}(X_f^\A).$$
The latter property is crucial in HMC to avoid the appearance of Jacobians in the Metropolis-Hastings  step, which are expensive to compute.


This suggests the following high-level strategy to sample using the volume-free $P$-preserving diffusions such that it maintains many of the geometric features that are key to the success of Hamiltonian-based Monte Carlo algorithms.
\begin{enumerate}
\item Begin by considering the complete recipe of $P$-preserving diffusions
$$ \diff Z_t  = \left( X^P_\A + P^{\sharp}(\gamma) +\half \div_{P}(Y_i)Y_i \right)\diff t +  Y_i(Z_t) \circ \diff W^i_t.$$
In general, it is not possible to obtain a computationally tractable expression for the solution to this system, since the solution must possess some symmetry for it to be tractable, whereas the target measure typically do not possess such symmetries. This leads us to the next step.
\item Decompose the target $P$ as $P \propto e^{-H} \mu_\M$, where $\mu_\M$ is an appropriate reference measure, for which tractable $\mu_\M$-preserving flows can be obtained. 
Hence, $\mu_\M$ is usually an invariant measure, such as the Lebesgue or Haar measure, or a simple probability measure, such as a Gaussian measure. The complexity of the target $P$ is then entirely contained within its density  $e^{-H}$.
We thus have the expression
\begin{align*}
 \diff Z_t &= \big(X_{H }^\A - \half  Y_i(H)Y_i\big) \diff t+\big( X^{\mu_\M}_\A + \half  \div_{\mu_\M}(Y_i)Y_i\big) \diff t \\
&\quad + P^{\sharp}(\gamma) \diff t + Y_i \circ \diff W^i_t,
\end{align*}
\item In order to reduce the complexity of the target density $e^{-H}$, we further split it into simpler components, $e^{-H} = \prod_j e^{-H_j}$. 
This is one of the important benefits associated to the lifting procedure used in HMC and the underdamped Langevin process, wherein the target density is complex, but its lift typically decomposes nicely into a \emph{potential} $V(q)$ and \emph{kinetic} $T(p)$ term,  both of which are simpler to handle in the lifted space since $V$ is $p$-independent and $T$ is $q$-independent.
Hence, when such simpler components do not exist, we must lift the process to some  phase space, i.e., vector bundle, over $\M$ where such decompositions exist.
\item Split the diffusion further into an $L^2(P)$-symmetric process
$$ \diff Z_t^{\SS} \defn  - \half\sum_j  Y_i(H_j)Y_i \,\diff t+ \half  \div_{\mu_\M}(Y_i)Y_i  \,\diff t +Y_i \circ \diff W^i_t,$$
and an $L^2(P)$-antisymmetric, deterministic process 
$$ \diff Z_t^{\A} \defn  \Big( \sum_j  X_{H_j }^\A  + X^{\mu_\M}_\A+  P^{\sharp}(\gamma) \Big) \diff t.$$
Hereafter, we will discard the topological obstruction term $P^{\sharp}(\gamma)$ for simplicity, as they are not necessary for measure-preservation.
\item Restrict the reversible component to volume-free processes, so that we get
$$
\diff Z_t^\SS = -  \half \sum_j  X_{H_j}^{\curl_{\mu_\M}(\A_i) \otimes \curl_{\mu_\M}(\A_i)}  \diff t+ \curl_{\mu_\M}(\A_i) \circ \diff W^i_t, $$
 which  ensures that the noise vector fields $Y_i =\curl_{\mu_\M}(\A_i)$ are volume-preserving, similar to the Langevin system.
On $\M=\R^n$, this can be implemented with an explicit integrator, as shown in \cite{bou2014metropolis} and on more general manifolds $\M$, the process can be lifted to a vector bundle over $\M$ and choosing the noise-fields to be {\em vertical}, the process will evolve purely on the fibres (which are vector spaces), where we can integrate this explicitly.

\item For the irreversible component,  setting $\A = \curl_{\mu_\M}(\V)$, the deterministic process simplifies to 
$$ \frac{\diff Z_t^{\A}}{\diff t} \defn   \sum_j  X_{H_j }^{\curl_{\mu_\M}(\V)},$$
which can be implemented with a  palindromic splitting integrator
 \cite{mclachlan2002splitting,hairer2006geometric},   that approximates the $P$-preserving flow of $X_H^{\curl_{\mu_\M}(\V)}$ with the composition of the flows of $X_{H_j}^{\curl_{\mu_\M}(\V)}$.
 Importantly, as discussed above, the choice $\A= \curl_{\mu_\M}(\V)$ guarantees that the splitting integrator will be volume-preserving.
Essentially, the splitting method used to compute the irreversible process  simplifies the implementation by decomposing the complicated target $e^{-H}\mu_\M$ into simpler targets $e^{-H_j}\mu_\M$, for which the corresponding flows are easier to construct, while still ensuring that the resulting numerical integrator is $\mu_\M$-preserving.
Here, again, if $H$ is too complex and does not have a sufficiently nice decomposition, it will be necessary to lift the process to an appropriate phase space where such decompositions exist.
In general, it is desirable to choose the  potentials $\{\mathcal V, (\A_i)\}$ that share the symmetries of the reference measure, in order for the tensor fields
$\curl_{\mu_\M}(\V)$ and $\curl_{\mu_\M}(\A_i)$ to inherit these symmetries.
\end{enumerate}

\begin{example}[Shadows]
Choosing $\A$ to be a Poisson structure with invariant measure $\mu_\M$ (such as those discussed in \cite{cruzeiro2018momentum,
weinstein1997modular,dufour2006poisson})
 further ensures that 
the splitting integrator used to integrate 
the irreversible deterministic component
will have a modified energy, called the shadow Hamiltonian, as a result of the Jacobi identity.
In other words, the numerical integrator will itself be a $\A$-Hamiltonian vector field with respect to the shadow Hamiltonian, and this feature is important to the success of HMC 
(see \cite{Leimkuhler2004,Kennedy:2012,sweet2009separable,
izaguirre2004shadow,radivojevic2019modified,
bou2018geometric}).
When the Poisson structure is constructed on a vector bundle over $\M$, such as the cotangent bundle, and the noise fields are chosen to be vertical, we obtain a natural generalisation  of \cite{arnaudon2019irreversible}, where an irreversible HMC algorithm on compact Lie groups was obtained, following the SOL-HMC construction in \cite{ottobre2016function} (see also \refsec{sec:manifold SOL}).

\end{example}

In general, the  Brownian motion $\diff W_t$ (resp. its Riemannian generalisation, discussed below) on $\R^n$ (resp. on a Riemannian manifold)  only preserves the Lebesgue measure $\dd x$ (resp. the Riemannian measure).
By choosing $H = 0$  and therefore $P=\mu_\M$, the diffusion \eqref{eq:vol-free-diffusion} reduces to 
$$ \diff Z_t =  \curl_{P}(\A_i) \circ \diff W^i_t,$$
which may be thought of as the general class of  \Definition{$P$-preserving Brownian motions}, wherein the Euclidean Brownian motion $\diff W^i_t$ is directed along $P$-preserving vector fields $ \curl_{P}(\A_i)$, to obtain a general measure-preserving Brownian motion.

\section{It\^o Diffusions with Riemannian Brownian Noise}\label{sec:examples-A-diff}

We now consider noise processes that are driven by Riemannian Brownian motion, which have powerful ergodic properties.
We say that a diffusion process is \emph{non-degenerate} when its generator is elliptic, and in this case, we can find a 
  Riemannian metric $\MM$ for which the generator takes the form
\begin{equation}\label{eq:generator-non-degenerate}
\L = X + \half \Delta, \quad \text{ for some drift } \quad X \in \mathfrak X(\M)\, ,
\end{equation}
where $\half \Delta$ is the {\em Laplace-Beltrami operator} associated with the metric $\MM$ (see \cite{ikeda2014stochastic,armstrong2016coordinate}).
The diffusion $B_t$ generated by $\half \Delta$ on $\M$ is called the 
\emph{Riemannian Brownian motion}  (see \cite{hsu1988brownian}), and such processes are 
 used for instance in the construction of stochastic gradient descent (SGD).
To get a glimpse of how this noise process is related to the Stratonovich noise discussed earlier, note that when the dissipative bracket $\S \defn Y_i \otimes Y_i \in \mathfrak{X}^2(\M)$ is positive definite, it defines a Riemannian co-metric on $\M$.
Conversely, given a Riemannian metric $\MM$ on $\M$, there always exists a local expansion of the co-metric $\MM^{-1}$ in terms of a
 finite set of vector fields $(Y_i)$, as $\MM^{-1} = Y_i \otimes Y_i$, since $\MM^{-1}$ is positive definite (this is analogous to taking the square-root of a positive definite matrix on $\R^n$).

The following theorem gives a full characterisation of $P$-preserving diffusions generated by \eqref{eq:generator-non-degenerate} (see \refsec{sec:non-degenerate-diffusions} for the proof).
\begin{theorem} Let $\vol$,  $\nabla$ and $\nabla \cdot$ be the Riemannian measure,  gradient and divergence respectively.
Any $P \propto p_\infty \vol$-preserving diffusion generated by \eqref{eq:generator-non-degenerate} takes for some $\A \in \mathfrak X^2(\M)$, the form (up to the usual topological obstruction term)
\begin{align}\label{eq-non-degenerate-general}
\begin{split}
\diff Z_t &=  \underbrace{- X_{\log p_\infty}^\A(Z_t)\diff t- \nabla \cdot \A(Z_t)\diff t}_{\text{irreversible drift}} \\
&\quad + \underbrace{\overbrace{ \half \nabla \log p_\infty(Z_t)\diff t}^{\text{Riemannian gradient flow}} + \overbrace{\diff B_t}^{\text{Riemannian Brownian motion}}}_{\text{reversible Riemannian overdamped Langevin system}}.
\end{split}
\end{align} 
\end{theorem}
We defer the discussion on the ergodicity of \eqref{eq-non-degenerate-general} on paracompact manifolds in \refsec{ergodicity}.
In the reversible case $\A=0$, 
the diffusion \eqref{eq-non-degenerate-general} gives us precisely the \emph{Riemannian overdamped Langevin equation}, used to construct the Riemann Metropolis-adjusted Langevin algorithm (MALA) when the Riemannian metric is obtained from an information divergence (see \cite{Girolami2011,Livingstone2014}).
On the other hand, the case $\A \neq 0$ is also of interest to us as it is well-known that the existence of an irreversible component can accelerate convergence to the target distribution, as demonstrated in \cite{hwang2005accelerating,duncan2017nonreversible,
rey2015irreversible} (a detailed analysis of the optimal drift for constant $\A$ is provided in \cite{lelievre2013optimal}). 
Alternatively, reversible overdamped Langevin systems ($\A=0$) with appropriate choices  of Riemannian metric $\MM$ can also  lead to accelerated convergence relative to the overdamped Langevin process \eqref{eq:overdamped-langevin-Rn}, as shown in \cite{abdulle2019accelerated}.

\begin{remark}[Reference Measures]
We point out that the Riemannian measure associated to the metric of the generator is sometimes  not an appropriate choice of reference measure.
This has in fact caused considerable confusion in the statistical literature when a target on Euclidean space is expressed in terms of the Riemannian measure instead of the Lebesgue measure 
 \cite{Simon:2013,liu2016stochastic,xifara2014langevin}.
 In that case, denoting by $p_\M$ the density of $P$ with respect to an appropriate reference measure $\mu_\M$ and $p_\infty$ the density associated with the Riemannian measure $\vol$, we can simply use the relation $\log p_\infty = \log p_\M - \log \frac{\dd \vol}{\dd \mu_\M}$ to convert \eqref{eq-non-degenerate-general} into a corresponding expression based on the measure $\mu_\M$.
For example on Euclidean space, it is well-known that the Riemannian Brownian motion $B_t$ can be expressed as
 $\diff B_t = \sqrt{\MM^{-1}}\diff W_t - \half \Gamma_\MM \,\diff t$, where $\Gamma_\MM^i \defn \MM^{jk}\Gamma^i_{jk} = \MM^{ir}\partial_r \log \sqrt{\MM} - \partial_r \MM^{ir}$ and $W_t$ is the standard Euclidean Brownian motion.
Further, if $\mu_\M = \dd x$, then $\dd \vol /\dd \mu_\M = \sqrt{\MM}$ and together with the explicit expression for $\diff B_t$, we can use this to express \eqref{eq-non-degenerate-general} in terms of the local Lebesgue density, as in \cite{xifara2014langevin}. One should note however that  $\mu_\M = \dd x$ is not a meaningful measure on general manifolds and therefore this expression only makes sense on Euclidean space.
\end{remark}

In local charts, one can also recover the Riemannian overdamped Langevin system directly from our $\A$-diffusion \eqref{eq:A-diffusion-detailed} as we show below.
First, let $(Y_i)$ be a family of vector fields on $\M$ such that $\S \defn Y_i \otimes Y_i$ is a positive definite tensor field. We then set $\MM \defn \S^{-1}$, which defines a Riemannian metric tensor. 
Now taking the reference measure $\mu_\M$ to be the Riemannian measure $\vol$ (locally, $\vol =\sqrt{|\MM|} \,\dd x$ where $|\MM|$ denotes the local determinant of $\MM$), the $\A$-diffusion \eqref{eq:A-diffusion-detailed} with $\A = 0$ becomes
\begin{align} \label{a-diffusion-zero}
\diff Z_t  = -\frac{\beta}{2} Y_i(H)Y_i \,\diff t + \half \div_{\vol}(Y_i) Y_i \,\diff t + Y_i \circ \diff W_t^i.
\end{align}
For the first term, we have $Y_i(H)Y_i  = X^{\MM^{-1}}_H = \nabla H$, so it is the Riemannian gradient of $H$. In local coordinates, the second term reads
\begin{align*}
\half \div_{\vol}(Y_i) Y_i^k &= \frac{1}{2\sqrt{|\MM|}}\frac{\partial}{\partial x^j}\left(\sqrt{|\MM|} Y_i^j\right)Y_i^k \\
&= \frac{1}{2\sqrt{|\MM|}}\frac{\partial}{\partial x^j}\left(\sqrt{|\MM|} Y_i^j Y_i^k\right) - \half Y_i^j \frac{\partial}{\partial x^j}Y_i^k,
\end{align*}
and  the random noise terms are related by the Stratonovich-to-It\^o correction
\begin{align*}
\underbrace{ Y_i^k \circ \diff W_t^i}_{\text{Stratonovich noise}} = 
\underbrace{\half Y_i^j \frac{\partial Y_i^k}{\partial x^j} \,\diff t}_{\text{Stratonovich-to-It\^o correction}} + \underbrace{ Y_i^k \,\diff W_t^i}_{\text{It\^o noise}}.
\end{align*}
Hence putting this together, we have
\begin{align*}
\half \div_{\vol}(Y_i) Y_i^k \,\diff t + Y_i^k \circ \diff W_t^i =  \frac{1}{2\sqrt{|\MM|}}\frac{\partial}{\partial x^j}\left(\sqrt{|\MM|} Y_i^j Y_i^k\right) \diff t + Y_i^k \,\diff W_t^i,
\end{align*}
which is precisely the local expression for the Riemannian Brownian motion $\diff B_t$ \cite{hsu2008brief}. Thus, in local charts, \eqref{a-diffusion-zero} becomes
\begin{align}
\diff Z_t = -\frac{\beta}{2} \nabla H\,\diff t + \diff B_t,
\end{align}
which is exactly the Riemannian overdamped Langevin system.

\begin{remark}[Riemannian Brownian Motion from the Orthonormal Frame Bundle]
We show here that the Riemannian Brownian motion can also be obtained globally as a projection of a Stratonovich diffusion defined on the orthonormal frame bundle.
Specifically, if we define the canonical horizontal vector field $L_i$ on the orthonormal frame bundle $\pi_O:O(\M)\to \M$, then the diffusion $\diff O_t = L_i \circ \diff W^i_t$ reduces to the Riemannian Brownian motion   \cite[theorem 4.2]{ikeda2014stochastic}. 
Moreover $L_i(\pi_O^*H)L_i$ is $\pi_O$-related to $\nabla H$, as follows from equations (4.12) and (4.22) in  \cite{ikeda2014stochastic}. Indeed, locally $L_i = e^r_i \partial_{q^r} -\Gamma^{\ell}_{sa}e^s_ie^a_b \partial_{e^{\ell}_b}$ where $(e_s)_s$ is an orthonormal frame, 
so $L_i(\pi_O^*H)L_i = e^r_i \partial_{q^r}(H)e^s_i \partial_{q^s}= \MM^{rs} \partial_{q^r}(H) \partial_{q^s}  $. 
Hence the Riemannian overdamped Langevin process 
$ \diff Z_t = -\frac{\beta}{2} \nabla H \diff t +\diff B_t,$
 is the projection under $\pi_O$ of the Stratonovich diffusions 
$ \diff Q_t = - \frac{\beta}{2} L_i(\pi^*_OH)L_i \diff t + L_i \circ \diff W_t^i$.

\end{remark}

\section{Deterministic Flow of Measure-preserving Diffusions on the Space of Volume Measures}\label{sec:flow on space of measures}

In this section, we describe the rate of change of functionals along the diffusion process.
The rate of change  of a  curve of volume measures $\mu_t$ along  a  diffusion process $Z_t$ is given by the forward Kolmogorov equation, 
$\frac{\partial \mu_t}{\partial t} = \L^* \mu_t$.
In particular, if $Z_t$ is an arbitrary  $P$-preserving diffusion, combining \eqref{eq:Fokker-adjoint-measures} and theorem  \ref{thm:completeness},  we find that the equation
\begin{align*}
\frac{\partial \mu_t}{\partial t} &= \div_{\mu_t}\left( -\curl_P(\A) - P^{\sharp}(\gamma)+\half(\div_{\mu_t}(Y_i)-\div_P(Y_i))Y_i \right) \mu_t
\end{align*}
describes the evolution  over the space of smooth measures of $\mu_t$ towards the stationary distribution $P$. 
Decomposing at each $t$ the target $P$ with respect to the reference measure $\mu_t$, as in \eqref{eq:modular-diffusion}, and using the fact $\curl \circ \curl =0$ we can simplify this expression to  
$$ \frac{\partial \mu_t}{\partial t} =\div_{\mu_t}\left( X^\A_{\log \frac{\dd P}{\dd \mu_t}} -\half Y_i \left(\log \frac{\dd P}{\dd \mu_t}\right) Y_i  \right) \mu_t -  \dd\frac{\dd \mu_t}{\dd P}(\mu^{\sharp}_t(\gamma)) \mu_t.$$
Let us for the moment ignore the topological obstruction for simplicity. Observe that 
$\mu_t$ satisfies the continuity equation 
$$ \frac{\partial \mu_t}{\partial t} +   \div_{\mu_t} \left(  \half X^{\SS  }_{\log \frac{\dd P}{\dd \mu_t}}- X^{\A  }_{\log \frac{\dd P}{\dd \mu_t}}\right)\mu_t =0.$$ 
Hence, the rate of change of KL divergence along the curve $\mu_t$ is
$$  \frac{\dd }{\dd t}\mathrm{KL}(\mu_t \| P) =-  \int \div_P \left( \half X^{\SS  }_{\log \frac{\dd P}{\dd \mu_t}}- X^{\A  }_{\log \frac{\dd P}{\dd \mu_t}} \right ) \mu_t, $$
and if  Stokes' theorem hold \footnote{that is $\int \dd i_{\half X^{\SS  }_{\log \frac{\dd P}{\dd \mu_t}}- X^{\A  }_{\log \frac{\dd P}{\dd \mu_t}}} \mu_t=0$}, we can further write (recall the notation \eqref{eq:bracket notation})
$$  \frac{\dd }{\dd t}\mathrm{KL}(\mu_t \| P) =-
 \half \int  \left\{\log \frac{\dd P}{\dd \mu_t},\log \frac{\dd P}{\dd \mu_t} \right\}_{ \S}  \, \mu_t  .$$
Since $\S$ is a dissipative bracket and $\mu_t$ is a smooth positive measure, this integral is non-negative for all $t$, and it follows that
$$  \frac{\dd }{\dd t}\mathrm{KL}(\mu_t \| P) \leq 0,$$
in concordance with the Euclidean case, see for example \cite{ma2019there}.

More generally, consider a   functional $F$ on the space of smooth measures.
Important families of  such functionals include the linear functionals
$$ F_f ( Q) \defn \metric{f}{Q}_* = \int f Q,$$
for some $f \in C^{\infty}(\M)$, and the functionals
$$ F^{P}_h (Q) \defn \int h\left(  \frac{\dd Q}{\dd P} \right)  Q,$$ 
parametrised by a choice of function $h: \R \to \R$,
which  include the KL divergence and other functionals that arise in a wide range of  applications \cite{khesin2008poisson,lott2006some,
weinstein1983hamiltonian,gangbo2010differential}.
The functional derivatives $\frac{\delta F}{\delta Q}$ with $F = F_f$ and $F_h^P$ are  $f$ and $ h\left(  \frac{\dd Q}{\dd P} \right)$ respectively. 
For any bracket $\bi$, we define the \emph{integral bracket} on the space of measures by  
$$ \{f,h\}_{\int_\bi}(Q) \defn \int \{f,h\}_\bi \, Q,$$
provided the integral converges (for example, $f$ or $h$ is compactly supported).
The following proposition shows that the integral thermodynamic bracket associated to the diffusion characterises the rate of change of $F$ along the process (see \refsec{sec:rate of change appendix} for the proof).
It can be used to optimize the brackets in the measure-preserving diffusion for the given task,
for example to improve the decay of a statistical divergence (see e.g. \cite{nielsen2020elementary}) along  the process to speed-up convergence to equilibrium.

\begin{proposition} Let $F$ be a  functional on the space of volume measures, and suppose $\frac{\delta F}{\delta Q}\in C_c^{\infty}(\M)$ (or  more generally that Stokes' theorem holds).
 The rate of change of $F$ along the $P$-preserving diffusion 
is then given by
$$ \frac{\dd }{ \dd t} F(\mu_t) = \left\{\log \frac{\dd P}{\dd \mu_t}, \frac{\delta F}{\delta \mu_t} \right\}_{\int_{\mathcal T}}(\mu_t) +  \metric{ \frac{\dd \mu_t}{\dd P} }{ \mu^{\sharp}_t(\gamma) \left[\frac{\delta F}{\delta \mu_t}\right]  }_{\mu_t}$$
where $\mathcal T$ is the thermodynamic bracket 
$ \SS/\sqrt 2  -\A$, and $\gamma$ the topological obstruction.
\end{proposition}

\section{Underdamped Langevin Diffusions on Manifolds}
\label{sec:manifold SOL}

As discussed in \refsec{sec:volume-free-diffusions}, an important feature common to both the underdamped and overdamped Langevin processes is the fact that they are both measure-preserving despite having no explicit reference measure contribution.
Therein we have derived the complete characterisation of these Langevin-like measure-preserving systems.

In this section, we will introduce another perspective regarding the underdamped Langevin process \eqref{Langevin-process-SDE}, and use this to construct irreversible Langevin-based MCMC samplers on manifolds. The idea is as follows.
Since the underdamped Langevin system evolves on a vector bundle $\pi:\F\to \M$ and we want the noise process to live entirely in the vertical direction, we can ask what are the measures $\mu_\F$ on $\F$ for which all vertical vector fields are $\mu_\F$-preserving, so that \emph{any} choice of vertical noise gives rise to a volume-free measure-preserving diffusion.
 This motivates the notion of a \Definition{Langevin pair},
defined as  a pair $(\A,\mu_{\F})$ such that
\begin{itemize} 
\item  for any function $f$, $X_f^\A$ is $\mu_\F$-preserving
\item  $\mu_{\F}$ is \Definition{horizontal}, i.e., $\L_Y \mu_{\F}=0$ for any vertical vector field $Y$. 
\end{itemize}
The $\A$-diffusion  generated by a Langevin pair $(\A,\mu_{\F})$ recovers the Langevin diffusions \eqref{Langevin-process-SDE} locally for any choice of vertical noise fields $(Y_i)$, with the deterministic Hamiltonian dynamics on $\R^n \times \R^n$ therein replaced by a more general $\A$-Hamiltonian vector field on $\F$.

As we will see below, when $\F = T\M$ is the tangent bundle over a Riemannian manifold,  choosing the noise to be vertical ensures that when $(\A,\mu_{T\M})$ is a Langevin pair and the Hamiltonian corresponds to a simple mechanical system, $H \defn \pi^*V+ \half \| \cdot \|^2$ (here $\| \cdot \|$ is the Riemannian norm), the thermostat process becomes an  OU process on the fibres, for which there is an explicit solution, and furthermore preserves the Gaussian distribution $e^{-\frac12 \MM(v, v)}$ on the fibres defined with respect to the Riemannian metric $\MM$.
An important example of a Langevin pair and target $\mu_H \defn e^{- \pi^*V - \half \| \cdot \|^2}\mu_{T\M}$
arises when the target $P$ on $\M$ is expressed  in terms of the Riemannian measure, $P=e^{-V} \vol$.
In this case, if 
$\omega_\flat^n$ denotes the Riemannian symplectic measure (associated to the symplectic structure $\omega_\flat \defn \flat^* \omega$ with $\flat$ the musical isomorphism), which in local tangent-lifted coordinate reads
$\omega_\flat^n = | \MM | \dd q \dd v$,
 then the pushforward of $\mu_H$ with $\mu_{T\M} = \omega^n_\flat$ is simply the target $\pi_* \mu_H = P$,
 so that any samples generated from a $\mu_H$-preserving process are transported under $\pi$ to samples from $P$.
Denoting by $\mathcal V$ the vertical lift of vector fields on $\M$ (which maps vector on $\M$ to vectors on $T\M$, see proof in \refsec{sec:underdamped Langevin manifolds} for the formal definition), we have the following result.

 \begin{theorem}\label{thrm:vertical-diffusion}
Suppose $(\A,\mu_{T\M})$ is a Langevin pair. 
If we choose the noise fields to be the vertical fields $Y_i\defn \mathcal V\circ X_i \circ \pi$ for $X_i \in \mathfrak X(\M)$, then the $\A$-diffusion generated by $e^{-H}\mu_{T\M}$ with $H \defn \pi^*V + \half \| \cdot \|^2$, reads
\begin{equation}\label{global-Vertical-SDE}
\diff (q_t,v_t) = \underbrace{X_{H}^\A(q_t,v_t)\diff t}_{\A-\text{Hamiltonian system}}- \underbrace{\half \metric{X_i(q_t)}{v_t}_{q_t}\mathcal V ( X_i(q_t)) \diff t}_\text{vertical kinetic Dissipation} +\underbrace{\mathcal V ( X_i(q_t)) \circ \diff W_t}_\text{vertical noise},
\end{equation}
where $\metric{u}{v}_q \defn \MM_q(u,v)$ is the Riemannian inner product at $q\in\M$, or in tangent-lifted coordinates,
 \begin{equation} \label{Vertical-SDE}
\diff (q_t,v_t) = \Big( X_{H}^\A(q_t,v_t)-\half M(q_t)v_t \Big) \diff t +\sigma(q_t) \circ \diff W_t,
\end{equation} 
where 
$$M(q)v=M(q)_{rj}v^j \partial_{v^r} \defn (\sigma \sigma^{\top}\MM(q))_{rj}v^j\partial_{v^r},$$ 
$$\sigma \defn \sigma_{ji}\partial_{v^j}  \defn (X_i)^j\partial_{v^j},
\quad \text{ and } \quad
\MM_{ij}(q) = \metric{\partial_{x^i}}{\partial_{x^j} }_q.$$
In particular $(\Pi,\omega^{n}_\flat)$ is a Langevin pair, where $\Pi$ is the Poisson bi-vector field associated to $\omega_\flat$.

 \end{theorem}
 
 When $\F$ is isomorphic to $T\M$, for example $\F=T^*\M$, then we can use the isomorphism to rewrite \eqref{global-Vertical-SDE}  over $\F$.
  In particular when $\M= \R^n$ and $T(q,v) = \half v^{\top} \MM  v$, using the musical isomorphism, the SDE \eqref{Vertical-SDE} becomes the Langevin dynamics on phase space, as seen in \eqref{Langevin-process-SDE} (with $\sqrt C \defn \MM  \sigma$, $ C \defn \MM  \sigma \sigma^{\top} \MM $)
 $$ \diff Q_t = \MM ^{-1} P_t \diff t, \qquad \diff P_t =  - \nabla V(Q_t)\diff t - \half C(Q_t) \MM ^{-1}(Q_t) P_t \diff t +\sqrt{C(Q_t)} \diff W_t,$$
where now, the definition $M\defn \sigma \sigma^{\top} \MM   $ in theorem \ref{thrm:vertical-diffusion} plays the role of the ``fluctuation-dissipation relation'' which ensures that the target $\propto e^{-H}  \diff q \diff p$ is preserved.
We thus see that, as claimed, \eqref{global-Vertical-SDE} is the manifold generalisation of the usual Langevin SDE, where the noise vector fields 
$X_i$ represent the columns of the ``vertical matrix'' $\sigma$, which only introduces randomness along the fibres (velocity), and
$\metric{X_i(q_t)}{v_t}_{q_t}$ describes the rate of change of the kinetic energy along the noise.
On $\F = \R^n \times \R^n$,  another example of a Langevin pair 
consists in choosing $\A$ to be a constant antisymmetric matrix (which is Poisson but not necessarily symplectic), and $\mu_\F$ the Lebesgue measure $\dd q \dd p$. 
It would be interesting to analyse the optimal properties of a subclass of these Langevin processes as proposed in~\cite{lelievre2013optimal}.

We can now  proceed to build various irreversible, Langevin-based MCMC schemes (which we abbreviate as iLMCMC) in a similar fashion to  \cite{horowitz1991generalized,
ottobre2016function,dobson2019reversible} 

\subsection{iLMCMC Algorithm}\label{iLMCMC}
Using vertical noise fields, we see that \eqref{global-Vertical-SDE} naturally decomposes into a Hamiltonian part and a vertical part which remains within the initial fibre, and thus only shifts the velocity (i.e., replaces the HMC heat bath).
We split \eqref{global-Vertical-SDE} into an $\A$-Hamiltonian part on $T\M$ 
\begin{equation}\label{Ham-iLMCMC}
 \frac{\diff z}{\diff t}= X_H^\A(z),
\end{equation} 
and an OU process within the tangent fibres $T_q\M$ (since the vertical lift is an isomorphism $T_v(T_q\M) \cong T_q \M$)
$$\dot q =0, \qquad \diff v_t = - \frac{\beta}{2} \metric{X_i(q)}{v_t}_{q}X_i(q) \diff t + X_i(q) \circ \diff W_t,$$
 which locally has the form\footnote{with $T_q \M\cong \R^n$ using the local basis, $M(q) \defn \sigma(q) \sigma^{\top}(q)\MM(q)$ where $\sigma_{ji}(q) \defn (X_i(q))^j$.}
\begin{equation}\label{OU-iLMCMC}
\dot q=0, \qquad \diff v_t = -\frac{\beta}{2}M(q)v_t  \diff t +\sigma(q) \circ \diff W_t,
\end{equation}
 and preserves the Gaussian $\propto e^{-\half  \metric vv_q} \diff v = \mathcal N(0, \MM^{-1}(q))$.
Thus, by choosing vertical noise fields $Y_i$ and a Langevin pair $(\A, \mu_\F)$ in the general $\A$-diffusion,  
we obtain a diffusion which splits naturally into an $\A$-Hamiltonian vector field
and a tractable OU-process in the fibres, as with the Euclidean case. 

Below, we consider the  bracket $\A= \Pi$ associated with the Riemannian symplectic structure, that is
$X^\Pi_f$ is the  Hamiltonian vector field of $f$ with respect to $\omega_\flat$,   and its geodesic integrators, and then consider the special case when $\M$ is a Lie group, where we recover a modified version of the algorithm presented in \cite{arnaudon2019irreversible} that is cheaper to compute.

\subsection{iLMCMC with Geodesic Integrators on Embedded Manifolds}

Suppose that $\iota:\M \hookrightarrow \R^k $ is an embedded manifold, equipped with a Riemannian metric that is defined by restricting the Euclidean metric to $\M$. 
We assume that we have (1) a tractable expression for the geodesic flow $\Phi^T$ of the kinetic energy $T \defn \half \| \cdot \|^2$, (2) a $C^1$-extension $W$ of the potential energy $V$ in the coordinates of the embedding, and (3) that the Riemannian metric corresponds to that used to define the reference measure of $P$ - otherwise we need to add a Radon--Nykodym term in our Hamiltonian.
If the geodesic flow is computationally intractable, convenient alternatives include the Riemannian integrators \cite{Leimkuhler:1996}, 
or RATTLE with reversibility check \cite{lelievre2019hybrid,leimkuhler2016efficient}.
Then we can apply a geodesic integrator which approximates \eqref{Ham-iLMCMC} by a composition of geodesic flow  $\Phi^T$,  and vertical gradient flow
generated by $X_V^{\Pi}$, whose integral curve starting from $(q_0,v_0) \in T\R^k|_\M$  reads
 $$q(t) = q_0 \qquad v(t) =  v_0- t \,\hor \left( \nabla_{q_0} W \right),$$
where $ \nabla_{q_0} W $ is the Euclidean gradient and $\hor$ is the orthogonal projection onto the tangent space of $\M$.
Similarly, to implement \eqref{OU-iLMCMC},
we need the  noise vector fields $X_i$ to be expressed in the coordinates of the embedding: 
that is, each noise field $X_i$ is given by a vector field $ b_i=(b^1_i,\ldots,b^k_i)$ on $T\R^k|_\M$ (i.e., $\partial \iota \circ X_i =  b_i \circ \iota$).
Then $ b_i$ defines the $i^{th}$ column of the matrix $\sigma$ and the process \eqref{OU-iLMCMC} is $\partial_q\iota$-related to the following process  on $\R^k$  (``$\cdot$'' is the dot product)
$$  \diff y_t = - \frac{\beta}{2}  b_i(q) \cdot y_t  b_i(q) \diff t + \sigma(q) \diff W_t.$$

Note that our potential energy does not include a $\log | \MM |$ term, unlike the geodesic MCMC of \cite{liu2016stochastic}. 
Indeed, as we discussed in the previous chapter, this terms does not give rise to a meaningful potential energy on manifolds.
In the very special case in which $\M=\R^k$ and we have  fixed a coordinate system (so $\R^k$ is no longer a manifold), then we can add the ``correction term''  $\log  |\M |$ to the potential energy in order to ensure that the algorithm generates samples from $e^{-V}\diff q$ rather than $e^{-V} \vol$, since typically, distributions on the Euclidean space are expressed in terms of the Lebesgue measure (see also \cite{holbrook2018note}).

\subsection{iLMCMC on Lie Groups}

Suppose now that the configuration space $\M=\G$ is a Lie group equipped with a left-invariant metric, $(\theta^i)$ is a set of Maurer--Cartan 1-forms,
and $\F$ is
 the symplectic manifold $\G\times \g$ \cite{barp2019hamiltonianB}.
The identity element of $\G$ will be denoted by $1$.

Let $v^i:\g \to \R$ defined by $v^i(\xi)= \theta^i_1(\xi)$ be the coordinates on $\g$, associated to the basis $(\xi_i)$ of the Lie algebra $\g$ dual to the Maurer--Cartan 1-forms (i.e., $\theta^i_1(\xi_j) = \delta^i_j$).
Since $T(\G \times \g)= T\G \oplus T\g$, the vector fields on $\G\times \g$ can be expanded as $X= a^i e_i + b^i \partial_{v^i}$, where $\partial_{v^i }\in \Gamma(T\g)$ and $e_i$ is the left-invariant vector field dual to $\theta^i$.
In \cite{arnaudon2019irreversible}, where the authors first derive an irreversible MCMC algorithm on Lie groups,
the  Hamiltonian fields were chosen to be of the form $Y_i \defn X_{U_i \circ \pi}$ for some noise potentials $U_i: \G \to \R$. 
Here, we will instead choose the noise fields to be $Y_i = \partial_{v^i }$ to make the computation of the OU process cheaper as we shall see (we could also have $Y_i = f_i(g) \partial_{v^i }$).
It follows from theorem \ref{thrm:vertical-diffusion} that $\div_{\omega^n}{Y_i} =0$ since $\partial_{v^i} = \mathcal V(\xi_i)$.\footnote{In fact we can check this directly without relying on local coordinates: writing the left Haar measure as $\Theta \defn \theta^1\wedge \cdots \wedge \theta^n$, we have
 $$ \mathcal L_{\partial_{v^i }}\omega^n= \big(\mathcal L_{\partial_{v^i }} \diff v \big) \wedge \pi^*  \Theta + \diff v \wedge \big( \mathcal L_{\partial_{v^i }} \pi^*\Theta \big). $$ 
 Now $\gen _{\partial_{v^i }} \diff v^j  = \diff \gen_{\partial_{v^i }} v^j= \diff i_{\partial_{v^i }} \diff v^j = \diff \delta_{ji}=0$.
 Then we have $\gen _{\partial_{v^i }} \pi^*\Theta=0$, since the flow $\Phi$ of $\partial_{v^i }$ is only non-trivial in the vertical direction.
Indeed its flow is $\Phi_t(g,v)=(g,v^1,\ldots, v^i +t,\ldots ,v^n)$ so
$$\big(\gen_{\partial_{v^i }} \pi^*\Theta \big)(g,v)= \frac{\diff}{ \diff t}\big|_{t=0} (\Phi^*_t(g,v) \pi^* \Theta)= \frac{\diff}{ \diff t}\big|_{t=0} \Big( \big(\pi \circ\Phi_t(g,v)\big)^* \Theta\Big)=0,$$ since $\pi \circ\Phi_t(g,v) = g$ is independent of $t$.}
Then, taking the kinetic energy $T \defn \half \MM_{ij} \theta^i_1 \otimes \theta^j_1$ on $\g$ associated to the  left-invariant metric with matrix $\MM$, \eqref{global-Vertical-SDE} 
becomes 
\begin{align}\label{OU-G}
\underbrace{\dot{q} =  v^j_t e_j(q_t)}_{\text{Reconstruction}}, \quad \diff v_t = \underbrace{ \ad^{\top}_{v_t}v_t \diff t}_{\text{Euler-Arnold}} \underbrace{- \MM^{jk}e_j(V)(q_t) \xi_k \diff t}_{\text{Potential flow}}  \underbrace{-\frac{\beta}{2} v_t \diff t + \diff W_t}_{\text{OU}},
\end{align}
where $(q_t, v_t) \in \G \times \g$. 
The Euler-Arnold term describes the geodesic motion of a Riemannian metric with symmetries (in this case left invariance), and  vanishes if the inner product on $\g$ is $\ad$-invariant \cite{holm1998euler,modin2010geodesics}.
In particular on SU$(3)$, the diffusion splits into the 
transition steps  used in the Hybrid Monte Carlo simulation for lattice QCD, with the OU process replacing the momentum heat bath.
In Euclidean space, we have $e_j = \xi_j = \partial_{x^j}$, and we recover the second order Langevin equation \cite{ottobre2016function,dobson2019reversible}.

It is particularly nice that the Ornstein-Uhlenbeck process $\diff v_t = -\half v_t \diff t + \diff W_t$ on $\g$ has an explicit solution given by
\begin{align}
v_{t+h} = e^{-\half h} v_t + \int^{t+h}_t e^{-\half (t+h-s)} \diff W_s,
\end{align}
with transition probability
\begin{align}
p(v_0, v) = \sqrt{\frac{1}{(2\pi)^{n}(1-e^{- h})}} \exp\left(-\frac{1}{2(1-e^{-\beta h})} \left\|v - e^{-\frac{1}{2}h} v_0\right\|^2 \right).
\end{align}
Hence, given an initial sample $(g_0,v_0)$, we obtain an irreversible MCMC algorithm on Lie groups by implementing the following steps:
\begin{enumerate}
	\item Solve the OU process exactly until time $h$ by sampling
	  \begin{align}
	    v^* \simeq \mathcal N \left(e^{-\half  h} v_0, (1-e^{- h}) \Id \right)\, , 
	  \end{align}
	  to obtain $(\bar{g}_0,\bar{v}_0) = (g_0, v^*)$;
	\item Solve the first-order Euler--Arnold equation, and approximate the Hamiltonian system  using $N$ leapfrog trajectories with step size $\dt >0$. 
	For example, for a matrix Lie groups with bi-invariant Riemannian metric, starting at $(\bar{g}_0,\bar{v}_0) = (g_0,v^*)$, we iterate \cite{barp2019hamiltonian}\\\\\
{\bf For} $k = 0,\ldots, N-1$: \footnote{For a non-matrix group, simply replace $\Tr\left(\partial_x V^T \bar{g}_k \xi_i\right)$ with $e_i|_g(V)$ \cite{barp2020bracket}.}
	      \begin{align*}
  		\bar{v}_{k+\half}&= \bar{v}_k -\frac{\dt}{2}\Tr\left(\partial_x V^T \bar{g}_k \xi_i\right) \xi_i \\
  		\bar{g}_{k+1} &= \bar{g}_k \exp\left(\dt \,\bar{v}_{k+\half} \right) \\
  		\bar{v}_{k+1} &= \bar{v}_{k+\half} -\frac{\dt}{2} \Tr \left(\partial_x V^T \bar{g}_{k+1} \xi_i\right) \xi_i
	      \end{align*}
		to obtain $(\bar{g}_N,\bar{v}_N)$. 
	\item Accept or reject the proposal by a Metropolis-Hastings step (although we note that implementing more advanced correction steps that take into account the whole trajectory is desirable \cite{betancourt2017conceptual}). We accept the proposal $(\bar{g}_N,\bar{v}_N)$ with probability
\begin{align*}
\alpha = \min \left\{1,\exp \left(-H(\bar{g}_N,\bar{v}_N) + H(\bar{g}_0,\bar{v}_0)\right)\right\}\, ,
\end{align*}
and set $(g_1,v_1) = (\bar{g}_N,\bar{v}_N)$. On the other hand, if the proposal is rejected, we set $(g_1,v_1) = (\bar{g}_0, -\bar{v}_0)$.
\end{enumerate}

Compared to the algorithm  presented in \cite{arnaudon2019irreversible}, our choice of noise field $Y_i \equiv \partial_{v_i}$ as opposed to $Y_i \defn X_{U_i \circ \pi}$ avoids having to compute matrix exponentials in the first step of the algorithm (which appears in the latter situation in the form $e^{-\half Dh}$, where $D$ is a matrix given by $D_i^j = Y_i^k Y_k^j$), thus significantly reducing the computational cost. 
On the other hand, choosing noise fields of the form $Y_i = X_{U_i \circ \pi}$ may be useful when the potential energies $U_i$ are adapted to the target distribution, for example by increasing the contribution of the noise at appropriate locations.

\section{Ergodicity of $\A$-Diffusions}\label{ergodicity}

Ergodicity plays an important role in many of the applications in which measure-preserving diffusions are employed, so here we discuss the conditions that the drift and diffusion vector fields $(X,Y_i)$ must satisfy in order to ensure unique ergodicity of the diffusion $\diff Z_t = X \,\diff t + Y_i \circ \diff W_t^i$. This will also allow us to clarify the conditions necessary for ergodicity in the Euclidean recipe \eqref{eq:intro-complete}.

First, let us denote by $Z_t^x$ the solution to the SDE $\diff Z_t = X \,\diff t + Y_i \circ \diff W_t^i$ with initial condition $Z_0 = x$.
In general, it is well-known that if the diffusion satifies:
\begin{enumerate}[ label= \bf(\roman*)]
\item the {\em strong Feller property}, that is if the Markov semigroup $T_t$ associated with the process $Z_t$\footnote{The Markov semigroup is defined as $T_t f(x) \defn \int_\M f(y) P_t(x,\dd y)$, where $P_t(x,A) \defn \mathbb{P}(Z_t^x \in U)$ for all $A \in \bi(\M)$.}  maps all bounded measurable functions into continuous functions, and \label{strong-feller}
\item {\em irreducibility}, that is, $P_t(x,U) \defn \mathbb{P}(Z_t^x \in U) > 0$ holds for any open set $U \subset \M$, any point $x \in \M$ and any $t > 0$, \label{irreducibility}
\end{enumerate}
 then it must be ergodic (see for example \cite{khas1960ergodic, glatt2014notes}).
By construction, $\A$-diffusions have a (strictly) positive invariant measure, and thus must be irreducible.
On the other hand, a sufficient condition to ensure that the strong Feller property \ref{strong-feller} holds, is   for instance that the vector fields $(X,Y_i)$ satisfy the  {\em H\"ormander condition}, that is, if the Lie algebra generated by $\{Y_i,[X,Y_i]: i=1,\ldots, N \}$ span the tangent spaces at every point (see \cite{arnaudon2010differentiation, glatt2014notes}), although this is a slightly stronger requirement than necessary when the vector fields are non-analytic (the condition is allowed to fail on appropriate hypersurfaces) \cite{bell2004stochastic,bell1995extension,
bismut1981martingales,cattiaux2002hypoelliptic,hormander1967hypoelliptic}. 

When the generator \eqref{eq:generator-sym-anti} is \emph{strongly elliptic},
that is, the Lie algebra generated by $(Y_i)$ spans the tangent spaces $T_x\M$ at every $x \in \M$, then
whenever a stationary measure exists, it must be unique on any paracompact, connected and orientable manifold $\M$, as shown in \cite[Proposition 6.1]{ichihara1974classification}.
In particular,  on compact manifolds $\M$, it is sufficient that the diffusion is {\em non-degenerate}, meaning that its generator is only elliptic (see \cite{elworthy1988geometric}). 
In Euclidean space, strong ellipticity means that there exist $C>0$ such that $\SS(\alpha,\alpha) \geq C \| \alpha \|^2$ for any covector $\alpha \in T_x^* \R^n$, where
 $\S = Y_i \otimes Y_i$. Thus, contrary to the claim made in the Euclidean recipe \cite[Theorem 1]{ma2015complete},
  it is not sufficient that $\SS$ is positive definite  (i.e., that the generator is elliptic) to ensure uniqueness of the target measure \cite{ichihara1974classification}. 

Unfortunately, many diffusions of  interest are not elliptic (for example those with vertical noise), so it would be interesting to check in future works what precise conditions on $\A$ and $Y_i$ are required for H\"ormander condition to hold in our $\A$-diffusion \eqref{eq:dissipative-form-A-diffusion-2}, given that $X= \curl_P(\A) +\half \curl_P(Y_i)Y_i$.
In particular, the subclass of Langevin-like  volume-free processes studied in \refsec{sec:volume-free-diffusions} is entirely made of rotationnels:
\begin{align}\label{eq:dissipative-form-A-diffusion-2}
\begin{split}
 \diff Z_t &= \left(X_H^{\curl_{\mu_\M}(\mathcal V)} - \half \beta X_H^{\curl_{\mu_\M}(\A_i) \otimes \curl_{\mu_\M}(\A_i)} \right) \diff t+ Y_i \circ \diff W^i_t \\
 &= -\curl_{\mu_\M}(i_{\dd H}\mathcal V) +\half \beta \curl_{\mu_\M}(i_{\dd H}\mathcal \A_i) \curl_{\mu_\M}(\mathcal \A_i)+  \curl_{\mu_\M}(\mathcal \A_i) \circ \diff W^i_t.
\end{split}
\end{align}
This insight may  be useful in tackling the ergodicity problem in the non-elliptic case, especially when combined with the fact $\curl_P$ has natural properties with respect to Lie brackets since it is a derivation of the Schouten--Nijenhuis bracket, a multi-vector generalisation of the Lie bracket \cite{koszul1985crochet}.

We should also note that the dynamical ergodicity of the underlying stochastic process is not sufficient for building algorithms that are robust enough for practical use -- it only provides an asymptotic guarantee for the behaviour of empirical averages of exact realizations of the stochastic process. The condition says nothing about the non-asymptotic behaviour of exact realizations nor any behaviour of the numerical discretizations to which we are limited to in practice. 
For example, MALA is obtained by applying an explicit Euler-Maruyama scheme to the overdamped Langevin process \eqref{eq:overdamped-langevin-Rn}, and composing it with an accept-reject step. However, it is well-known that the algorithm does not maintain the ergodicity properties of the continuous process  that it is derived from (see \cite{roberts1996exponential} for example).

To guarantee that the algorithm will be useful in practice, we need to bound the convergence of the discretized stochastic process towards its asymptotic limit, if one exists.  Recent progress has been made in understanding the converge of both exact and discretized Langevin diffusions in Wasserstein distances \cite{durmus2019analysis,cheng2018sharp,
raginsky2017non,eberle2019couplings}, although these results are limited to sufficiently nice target distributions that limit their practical utility. 

Classic statistical results do not consider the convergence of diffusions themselves, but rather the Metropolis-Hastings transitions that use the discretized diffusions as a proposal distribution.  Using coupling techniques, they demonstrate when the convergence admits geometric bounds in the total variation distance. Although these bounds are not particularly tight, they ensure the existence of central limit theorems which then allow for the convergence to be estimated well empirically.

When moving beyond diffusions to more general second-order Markov processes the problem becomes even harder.  The limited theoretical results \cite{durmus2017convergence,livingstone2016geometric} focus largely on necessary conditions for geometric bounds in the total variation distance and hence the existence of central limit theorems.  Although these conditions are not sufficient to guarantee any particular non-asymptotic behavior, they motivate empirical diagnostics that help practitioners identify target distributions beyond the scope of the algorithm.

\section{A Brief History of Measure-Preserving Diffusions}\label{sec:history-P-diffusions}

The history of measure-preserving processes is a long one, and in this section, we only aim to provide a handful of previous works that are directly related to this one. 
While in the machine learning community, the characterisation of measure-preserving diffusions on $\R^n$ was popularised in the recent NeurIPS article  
\cite{ma2015complete}, 
anterior closely related results can be found in the SDE literature.
For example in 1977, Robert Graham discusses the covariance of the Fokker--Planck equation, in the context of non-degenerate diffusions, and uses the Riemannian metric associated to the noise  to define a Riemannian divergence which allows him to differentiate second-order tensors  \cite{graham1977covariant} (a nice discussion of the work of Graham is also provided in \cite{eyink1996hydrodynamics}).
By analogy with Maxwell's equations, Graham notes that the Fokker--Planck current must be the Riemannian divergence of some anti-symmetric tensor field, which is precisely the result provided in \cite[Theorem 2]{ma2015complete}. 
The covariance of the Fokker--Planck equation and the diffusion process is also discussed in \cite{batrouni1986coordinate} and \cite{masoliver1987geometrical}, where the latter article derives conditions 
for the diffusion to be reversible (see also \cite{nelson1958adjoint}).
A less intuitive characterisation of the Fokker--Planck current of measure-preserving diffusions is also given in \cite{hwang1993accelerating} and in the case of non-degenerate diffusions on compact oriented manifolds, the book \cite{ikeda2014stochastic} effectively derives a complete recipe using the Riemannian metric derived from the noise to transform the Fokker--Planck vector field into a 1-form, that is then studied via its Hodge--de Rham decomposition.

A major shortcoming of these references is that they all assume the noise to be non-degenerate, in order to equip the manifold with a Riemannian metric, as well as the orientability of the manifold to work with differential forms.
 Yet, many important measure-preserving diffusions are degenerate, such as (the deterministic) Hamiltonian systems, or even the underdamped Langevin process. 
While on Euclidean space, we have a ``natural'' metric that we can use to differentiate second-order tensors, such metrics do not exist on general manifolds.
Concurrently to this article, a  covariant formulation was introduced in \cite{ding2020covariant} to remove the dependence on the metric.
However contrary to our recipe, this work relies on local coordinates, and does not take into account the presence of topological obstruction, thus leading to a recipe which is only valid locally, since on manifolds, there are divergence-free vector fields that cannot be globally expressed as the $P$-rotationnel of a bi-vector field.
Moreover, the relation with the canonical geometry of $P$ is not shown,
and as we have illustrated in \refsec{sec:volume-free-diffusions}, our geometric framework offers new insights even on Euclidean space.
More importantly, the intrinsic geometry of $P$ allows us to re-contextualise  the theory of measure-preserving diffusion within the realm of differential geometry.

In the bigger picture, we see that these results are a combination of two things: 
the construction of the Fokker--Planck equation \cite{fokker1914mittlere,planck1917satz,
kolmogoroff1931analytischen,risken1996fokker} and
 its expression in geometric 
form as discussed above (see also \cite{osada1987diffusion}),
 combined with  geometric characterisations of divergence-free vector field.
 Such characterisations have already been studied in several works on geometric integrators  \cite{mclachlan2002splitting,hairer2006geometric} and have been known for at least a century, as seen in the work \cite{de1931analysis}.
 For example, the fact that vector fields that preserve the Lebesgue measure can be written as the divergence of an antisymmetric matrix (without any integrability assumption), which is all that is needed to obtain \cite[Theorem. 2]{ma2015complete},
  goes back at least to the works of Poincaré and Volterra in the 1880s, e.g.
 \cite{volterra1889generalisation}.

\section{Conclusion}

In this work, building on from results in Poisson mechanics, geometry, topology, physics, and statistics, we have 
presented the  complete and canonical characterisation of measure-preserving diffusions on arbitrary manifolds, 
which play a central role in many areas of science, both in terms of mathematical modelling and in statistics/machine learning. Our general framework provides a sound mathematical basis to design and study them. 
It not only extends and contextualise the results obtained in \cite{ma2015complete}   for the Euclidean case, and improves it by removing the integrability constraint, 
 but more importantly provides an elegant interpretation from a purely topological standpoint, relying solely on the geometry of the volume manifold $(\M, P)$.
This is achieved by constructing {\em potentials} for the Fokker-Planck current in the same way as how physical `potentials' such as the potential energy and the magnetic potentials are constructed in classical mechanics. On contractible sample spaces, the resulting diffusion is specified, just as with thermodynamic systems,  by two `brackets': an antisymmetric $\A$ that presents itself as the `potential' for the Fokker-Planck current, and a dissipative one $\S$, that is generated by the noise vector fields $(Y_i)$. 
Moreover, when the topology of the manifold is non-trivial (e.g., it is not connected), we also need to take into account an extra topological obstruction term to achieve global completeness, which we have shown to be parametrised  by harmonic forms on compact orientable manifolds, and non-zero elements of a twisted de Rham cohomology group in general.

In addition to fully characterising the measure-preserving diffusions, 
we have also studied their reversibility, associated flows on the space of volume measures,
the generalisations of Langevin processes to manifolds, 
and introduced a new recipe for volume-free  diffusions that are well-suited to the construction of Langevin-like sampling algorithms.
Our canonical formulation  properly takes into account the critical assumptions on the target measure, namely, that it is smooth and globally supported,
which allows us to analyse the diffusions through measure-informed versions of known results in differential geometry (obtained using  isomorphisms  induced by the target).
From a practical point of view, having intrinsic results  that  focus on the target measure and do not make any extra assumptions imply that these can be applied  regardless of the particular application, such as physics and machine learning.
In future works, we will further address how to develop efficient MCMC algorithms to sample from manifolds using this complete recipe, and furthermore, we aim to extend this framework in the context of infinite-dimensional diffusions, as considered in \cite{beskos2009mcmc,law2014proposals,ottobre2016function}, which may be useful for applications in stochastic climate modelling and data assimilation \cite{bouchet2012statistical,franzke2015stochastic,law2015data,majda2010linear,majda2008applied,
majda2001mathematical}.



\subsection*{Acknowledgements} We would like to acknowledge support for this project
from the National Science Foundation (NSF grant IIS-9988642)
and the Multidisciplinary Research Program of the Department
of Defense (MURI N00014-00-1-0637). 


\newpage

\appendix
\label{app:theorem}



\section{Proofs}

\subsection{Derivation of Fokker--Planck operator} \label{sec:proof-adjoint}

\begin{lemma} \label{proof-adjoint-gen} Given smooth vector fields $X, Y_1, \ldots, Y_N$ on  $\M$, consider the Stratonovich SDE
\begin{equation}
 \diff Z_t = X(Z_t) \,\diff t +  Y_i(Z_t) \circ \diff W^i_t\,.
\end{equation}
Then its generator is given by
$
\gen f = Xf+\half  Y_iY_if
$.
The Fokker--Planck operator, viewed as the formal adjoint of $\L$ in $L^2(\mu_\M)$ (if the boundary $\partial \M$ is non-empty we restrict to functions that vanish on the boundary) is given  by
\begin{equation}  
  \gen^*f =    \div_{\mu_\M}\big(-fX + \half \div_{\mu_\M}(fY_i)Y_i \big)=\div_{\mu_\M}\left(-fX + \half Y_i(f)Y_i+ \half  f\div_{\mu_\M}(Y_i)Y_i \right) \, .
\end{equation} 
\end{lemma}

\begin{proof}
Vector fields satisfy Leibniz rule
\begin{align*}
 \int g X(f) \mu_{\M} = \int X(gf)\mu_{\M} -\int fX(g) \mu_{\M}, 
\end{align*}
and note that $\int_\M \div(gX) \mu_{\M}=\int_\M \diff i_{gX}\mu_{\M} = \int_{\partial \M} i_{gX} \mu_{\M}=\int_{\partial \M} gi_{X} \mu_{\M}$ which vanishes if $g|_{\partial M}=0$ (here $\div \defn \div_{\mu_\M}$).
Hence
\begin{align*}
 \int g X(f) \mu_{\M} & = -\int gf \div(X)\mu_{\M} -\int fX(g) \mu_{\M} \\
 &=  \int f \big( -g\div(X)-X(g)\big) \mu_{\M} 
 \\
  &=  \int f \big( -\div(gX)\big) \mu_{\M} = \int f \gen^*g \mu_{\M}.
\end{align*}

Now let us compute the adjoint of the diffusion component. 
 First note that 
\begin{align*}
Y_k(g Y_kf) &= \div \big( gY_kf Y_k\big)-gY_kf \div(Y_k) = \\
&= \div \big( gY_kf Y_k\big)-\div \big( f g \div(Y_k) Y_k\big)+f \div\big(g \div(Y_k)Y_k \big) 
\end{align*} 
and
\begin{align*}
Y_igY_if= \div(f Y_ig Y_i)-f\div(Y_ig Y_i).
\end{align*}
It follows that if $f,g$ vanish on the boundary
\begin{align*}
\half\int g Y_i Y_if \mu_{\M}  &= \half\int \Big(Y_i\big(g Y_i f\big)  -  Y_i g Y_i f \Big)\mu_{\M}  \\
&=  \half
\int \Big( f \div\big(g \div(Y_i)Y_i \big) +f \div(Y_ig Y_i)\Big)\mu_{\M}  \\
&= \int f \half\div \big( \div(gY_i)Y_i \big) \mu_{\M}  = \int f  \gen^*g \mu_{\M}  .
\end{align*}
\end{proof}

%

\subsection{Proof of Theorem \ref{thrm-gibbs}} \label{sec:thrm-gibbs}

\begin{theorem} 
The Gibbs measure \eqref{gibbs-def} is preserved under the bracket diffusion \eqref{Bracket-SDE}, in the sense that  $\gen^*p_{\infty}=0$,
if and only if the vector field $Y$ satisfies 
\begin{equation}
  \div_{\mu_\M}(X_{p_{\infty}}^\bi-\beta p_{\infty}Y)=0\, .
  \end{equation}
\end{theorem} 
\begin{proof}
Let $p(x)= \frac 1 Z e^{-\beta x}$ for $x \in \mathbb R$, so $p_{\infty} = p \circ H :\M \to \R $, and 
$X_{p_\infty}=-p_\infty X_H$. 
Then using the Fokker-Planck operator defined in \eqref{Fokker-general}, we have
\begin{align*}
\gen^*p_{\infty}&=\div\Bigg( -p_{\infty} \Big(X_{H }-\frac{\beta}{2}Y_i(H )Y_i +\half\div(Y_i)Y_i+Y \Big)+ \half \div\big(p_{\infty}Y_i\big)Y_i \Bigg) \\
&=\div\Bigg(  \frac 1 {\beta}X_{p_{\infty}}-p_{\infty}Y-\frac{1}{2}Y_i(p_{\infty})Y_i- \half p_{\infty}\div(Y_i)Y_i + \half \div\big(p_{\infty} Y_i\big)Y_i \Bigg)\\  
&=\div\Bigg( \frac 1 {\beta}X_{p_{\infty}}-p_{\infty}Y -\frac{1}{2}Y_i(p_{\infty})Y_i-\half p_{\infty}\div(Y_i)Y_i + \half \div\big(p_{\infty}Y_i\big)Y_i \Bigg) 
 \\
&=\div\Bigg(\frac 1 {\beta}X_{p_{\infty}}-p_{\infty}Y  -\half \div\big(p_{\infty}Y_i\big)Y_i + \half \div\big(p_{\infty}Y_i\big)Y_i \Bigg) \\
&=\div\Big(\frac 1 {\beta}X_{p_{\infty}}-p_{\infty}Y \Big),
\end{align*}
thus $\gen^*p_{\infty}=0$ iff $Y$ satisfies $\div\big(\frac 1 {\beta}X_{p_{\infty}}-p_{\infty}Y \big)=0$.
\end{proof}

\subsection{Recovering The Euclidean Complete Recipe}\label{sec:complete-recipe-recovery}

\begin{corollary} Let $\M=\R^n$ with flat metric, $\mu_\M = \diff x$,  $\sigma_{ij}\defn Y^i_j$, $D \defn\half \sigma \sigma^T$ and $Q_{ij} \defn \A^{ij}$. 
Then \eqref{Bracket-SDE} reduces to the It\^o diffusion considered in the complete recipe for SGMCMC \cite{ma2015complete}
\begin{equation}
\diff Z_t = -\big( Q+D\big)\nabla H \diff t + \nabla \cdot \big(Q+D \big)\diff t + \sqrt{2D}\diff W_t\,,
\end{equation}
where  we use the convention $(\nabla \cdot Q)_i \defn \partial_j Q_{ij}$ as in   \cite{ma2015complete}.
\end{corollary}
\begin{proof}
In that case, the term $Y_i(H )Y_i= Y^j_iY^k_i \partial_jH  \partial_k$ represents the ``symmetric semi-definite part" of the dynamics since in flat space this is
 $(\sigma\sigma^T)_{jk} \partial_jH  \partial_k$ where $\sigma_{ij}=\sigma^i_j=Y^i_j$,
while the term $\div(Y_i)Y_i=\sigma_{ki}\partial_j \sigma_{ji}\partial_k$
 together with the It\^o-to-Statonovich correction yields the divergence of the diffusion matrix $D \defn \half \sigma \sigma^T$.
 Using the convention $(\nabla \cdot Q)_i \defn \partial_j Q_{ij}$ in \cite{ma2015complete}
\begin{multline*}
 \mathrm d Z_t = \left(\A ^{ji}\partial_j H \partial_i - \frac{\beta}{2} (\sigma\sigma^T)_{jk} \partial_jH  \partial_k +\half \sigma_{ki}\partial_j \sigma_{ji}\partial_k 
 - \partial_j \A ^{ji}\partial_i\right) \mathrm{d} t \\+ \sigma_{ij} \partial_i \circ \mathrm{d}W^j_t \\
 =
 \left(\A ^{ji}\partial_j H \partial_i- \frac{\beta}{2} (\sigma\sigma^T)_{jk} \partial_jH  \partial_k +\half \sigma_{ki}\partial_j \sigma_{ji}\partial_k 
+\half \sigma_{ki} \partial_k \sigma_{ji} \partial_j + \partial_j \A ^{ij}\partial_i\right) \mathrm{d} t \\+ \sigma_{ij} \partial_i  \mathrm{d} W^j_t \\
= \left( -Q \nabla H - \beta D \nabla H + \nabla \cdot D + \nabla \cdot Q \right) \mathrm{d} t + \sqrt{2D} \mathrm{d} W_t.
\end{multline*}
\end{proof}

\subsection{Derivation that Curl is Modular Field}\label{sec:curl-is-modular}
This result was proved in the thesis of one of the authors \cite{barp2020bracket}, but the proof relies on introducing Schouten--Nijenhuis brackets.
 Thus we here include a new direct and  more  constructive  proof:
\begin{proof}
We will prove the equivalence by demonstrating that the action of the modular vector field and the curl vector field are equal for all smooth functions $f$,
\begin{equation*}
\curl_P(\A)(f) \cdot P = X^{P}_{\A}(f) \cdot P.
\end{equation*}
The curl vector field is defined implicitly by the action
\begin{equation*}
i_{\curl_P(\A)} P = \dd \left( P^{\flat}(\A) \right)
\end{equation*}
which implies that for any $f$
\begin{equation} \label{eqn:curl_def}
\dd f \wedge i_{\curl_P(\A)} P = \dd f \wedge \dd \left( P^{\flat}(\A) \right).
\end{equation}
To simplify the left hand side note that
\begin{align*}
0 
&=
i_{\curl_P(\A)} \left( \dd f \wedge P \right)
\\
&=
i_{\curl_P(\A)}(\dd f) \cdot P - \dd f \wedge i_{\curl_P(\A)} P
\\
&=
\curl_P(\A)(f) \cdot P - \dd f \wedge i_{\curl_P(\A)} P,
\end{align*}
or
\begin{equation*}
\dd f \wedge i_{\curl_P(\A)} P =\curl_P(\A)(f) \cdot P.
\end{equation*}
For the right hand side we take
\begin{align*}
\dd \left( \dd f \wedge (P^{\flat}(\A)) \right)
&=
\dd^{2} f \wedge P^{\flat}(\A) - \dd f \wedge \dd \left( P^{\flat}(\A) \right)
\\
&=
-\dd f \wedge \dd \left(P^{\flat}(\A) \right).
\end{align*}
Substituting both results in Equation \eqref{eqn:curl_def} then gives
\begin{equation} \label{eqn:curl_action}
\curl_P(\A)(f) \cdot P = -\dd \left( \dd f \wedge P^{\flat}(\A) \right).
\end{equation}
Now we use the fact that for any function 
$P^{\sharp}\left(\dd f \wedge P^{\flat}(\A) \right) = i_{\dd  f} \A$ \cite{dufour2006poisson},
so that equation \eqref{eqn:curl_action} becomes
\begin{align*}
\curl_P(\A)(f) \cdot P 
= -\dd P^{\flat}\left( i_{\dd f} \A \right)
=
\dd P^{\flat}\left( X_f^\A \right)
=
 \div_{P}(X^{\A}_{f}) \cdot P
\end{align*}
as desired.
\end{proof}

\subsection{Derivation of Adjoint of Integration Pairing} \label{sec:fokker-density-adjoint}

\begin{lemma}  The Fokker--Planck operator of the Stratonovich SDE \eqref{general-SDE}, viewed as the
 formal adjoint of $\gen$ with respect to the pairing $(f,P) \mapsto \int f \dd P$
between smooth, compactly supported functions $f$ and smooth measures $P$, is given by
\begin{equation}
\L^*P = -\L_XP + \half \L_{Y_i} \L_{Y_i}P.
\end{equation} 
\end{lemma}
\begin{proof}
Recall the generator is $\L = X + \half Y_i \circ Y_i$.
Thus for the deterministic drift we find
\begin{align*}
\int X(f) P &= \int \dd f(X) P=
\int  i_X (\dd f) P
=\int \dd f \wedge i_X P \\
&=  \int \dd ( f \wedge i_X P)-\int f \dd i_X P= - \int f \L_X P. 
\end{align*}
Similarly, for the diffusion coefficient, 
\begin{align*}
 \int Y_i Y_i(f) P &= - \int Y_i(f) \L_{Y_i}P = \int f \L_{Y_i} \L_{Y_i}P.
\end{align*}

\end{proof}

\subsection{Derivation of $\A$-diffusion expressed in terms of the reference measure $\mu_\M$}\label{sec:derivation-of-canonical-A-diffusion}
Here, we include additional details to the derivation of the $\A$-diffusion \eqref{eq:modular-diffusion}  in terms of reference measure $\mu_\M$ from the abstract $\A$-diffusion \eqref{eq:A-diff-theorem}  expressed in terms of the target $P = e^{-\beta H} \mu_\M$.
First note that if $f$ is a non-vanishing  function, then $(f\mu_\M)^{\flat} = f \mu_\M^{\flat}$ and so $(f\mu_\M)^{\sharp} = \mu_\M^{\sharp} \circ \frac1 f$.
Hence
\begin{align*}
 \curl_{f\mu_\M}(X) &= (f \mu_\M)^{\sharp} \circ \dd \circ (f \mu_\M)^{\flat}(X) = (\mu_\M)^{\sharp} \left( \frac1 f \,\dd \left( f \,(\mu_\M)^{\flat}(X)\right)\right) \\
&= (\mu_\M)^{\sharp} \circ \dd \circ (\mu_\M)^{\flat}(X) + (\mu_\M)^{\sharp} \left( \left(\frac1 f \,\dd f\right) \wedge (\mu_\M)^{\flat}(X)\right)\\
&=
(\mu_\M)^{\sharp} \circ \dd_f \circ (\mu_\M)^{\flat},
\end{align*}
where $\dd_f \defn \dd + \dd \log | f| \wedge$ is the distorted de Rham derivative.

\begin{proposition}
Given $P = e^{-\beta H} \mu_\M$, we have
\begin{equation}\label{eq:div-P-Y}
\div_P(Y) = \div_{\mu_\M}(Y) - \beta Y(H),
\end{equation}
for any vector field $Y \in \X(\M)$ and
\begin{equation}\label{eq:curl-P-A}
\curl_P(\A) = X_\A^{\mu_\M} + \beta X_H^\A,
\end{equation}
for any bi-vector field $\A \in \X^2(\M)$.
\end{proposition}

\begin{proof}
Taking $f = e^{-\beta H}$ in the above, we have
\begin{align*}
\div_P(Y) = \curl_P(Y) &= (\mu_\M)^{\sharp} \circ \dd \circ (\mu_\M)^\flat(Y) - \beta (\mu_\M)^{\sharp} \left( \dd H \wedge (\mu_\M)^{\flat}(Y)\right) \\
&= \div_{\mu_\M}(Y) -\beta (\mu_\M)^{\sharp} \left( \dd H \wedge i_{Y} \mu_\M \right).
\end{align*}
Now, using the identity $i_{Y}(\dd H \wedge \mu_\M) = i_{Y}\dd H \wedge \mu_\M - \dd H \wedge i_{Y} \mu_\M$ and noting that $\dd H \wedge \mu_\M = 0$ (since $\mu_\M$ is a top degree twisted form on $\M$), we have $\dd H \wedge i_{Y} \mu_\M = i_{Y}\dd H \wedge \mu_\M = (\mu_\M)^\flat (Y(H))$. The first identity \eqref{eq:div-P-Y} then follows immediately.

Similarly, we have
\begin{align*}
\curl_P(\A) &= \curl_{\mu_\M}(\A) - \beta (\mu_\M)^{\sharp} \left( \dd H \wedge (\mu_\M)^{\flat}(\A)\right).
\end{align*}
Noting that $(\mu_\M)^{\sharp} \left( \dd H \wedge (\mu_\M)^{\flat}(\A)\right) = i_{\dd H} \A \defn - X_H^\A$ (see \cite{dufour2006poisson}) and $\curl_{\mu_\M}(\A) = X_\A^{\mu_\M}$ (see \refsec{sec:curl-is-modular}), the second identity \eqref{eq:curl-P-A} follows.
\end{proof}

\subsection{Derivation of Measure-Preserving Diffusion on Riemannian Manifolds}\label{sec:Riemannian-proof}

\begin{theorem}
Let $\M$ be a compact orientable Riemannian manifold and let $\mu_\M \defn \vol$ and $\nabla \cdot$ denote respectively the Riemannian measure and divergence.
Then, any $e^{-H}\vol$-preserving diffusion has the form
 \begin{equation}
\diff  Z_t \defn  \left(X_H  + \half \left(    \nabla \cdot Y_i -Y_i(H) \right)Y_i -   \nabla \cdot \A  +  e^{H} \sharp \star ^{-1}\zeta  \right) \diff t + Y_i \circ \diff W^i_t,
\end{equation}
where $\star$ is the Hodge star operator, $\sharp$ is the Riemannian musical isomorphism, $\A$ is an antisymmetric bracket and $\zeta$ is a harmonic $(n-1)$-form (i.e., it satisfies ``Maxwell's equations'' $\diff \zeta=0$, $\diff \star \zeta =0$).
\end{theorem}

\begin{proof}
Let $\delta \propto \star \dd \star$ be the co-differential. Since $\div_{\vol} = \delta \circ \flat$, the Fokker-Planck operator \eqref{Fokker-general} may be written as
\begin{equation}
  \gen^*g =  \delta\left(-gX + \half \nabla \cdot(gY_i)Y_i \right)^\flat  \, .
\end{equation}
It then follows by the Hodge decomposition that
$\L^*p_{\infty} =0$ iff $ \delta\left(-p_{\infty}X + \half  \nabla \cdot(p_{\infty}Y_i)Y_i \right)^\flat=0$ iff
$ \diff \star \left(-p_{\infty}X + \half  \nabla \cdot(p_{\infty}Y_i)Y_i \right)^\flat=0$ iff
$$ \star \left(-p_{\infty}X + \half  \nabla \cdot(p_{\infty}Y_i)Y_i \right)^\flat = \diff \alpha - \gamma $$
where $\gamma$ is Harmonic,
that is co-exact and closed, and vanishes whenever the $\dim(\M)-1$ de Rham cohomology is trivial.
Thus
$\left(-p_{\infty}X + \half  \nabla \cdot(p_{\infty}Y_i)Y_i \right)^\flat = \delta \varepsilon - \star ^{-1}\gamma $
where $\varepsilon \defn \star^{-1} \alpha$ is a 2-form.
Hence,
$$ -p_{\infty}X + \half  \nabla \cdot(p_{\infty}Y_i)Y_i =  \sharp \delta \varepsilon+ \sharp \star ^{-1}\gamma =  \nabla \cdot (\A  p_{\infty}) -  \sharp \star ^{-1}\gamma\, ,$$
for an appropriate antisymmetric bracket $\A $.
Now, $\nabla \cdot (\A  p_{\infty}) = p_{\infty} \nabla \cdot \A  - p_{\infty} X_H$, since $\nabla \cdot \A= \curl_{\vol}(\A)$ (proved below), and for any $f$
\begin{align*}
\curl_\mu(p_\infty\A ) [f] &= - \div_\mu( p_\infty X_f)=- X_f(p_\infty) - p_\infty \div_\mu X_f =
+ p_\infty X_f(H)- p_\infty \div_\mu X_f\\
&=
- p_\infty X_H(f)- p_\infty \div_\mu X_f,
\end{align*}
 so $$ \curl_\mu(p_\infty\A )= -p_\infty X_H+ p_\infty \curl_\mu \A . $$
 It follows that the drift vector field must take the form
$$
 X =  X_H + \half  \nabla \cdot(Y_i)Y_i - \half Y_i(H) Y_i  -    \nabla \cdot \A   +  e^H \sharp \star ^{-1}\gamma.
 $$
 \end{proof}

Note that the Riemann divergence $\nabla \cdot \A$ in the previous theorem is precisely the curl of the Riemann measure (up to a sign).
Indeed, if $(\M,\MM)$ is a pseudo-Riemannian manifold, recall that $\nabla \cdot \A= \text{Tr} \nabla \A $, where $\text{Tr}$ is the trace and $\nabla$ the covariant derivative.
Then using the fact that $S_{ab} \A^{ab}=0$ for any symmetric tensor $S_{ab}$ and anti-symmetric tensor $\A^{ab}$, we see that for any smooth function $f$
\begin{align*}
\nabla \cdot \big( f \nabla \cdot \A\big)&=
\frac{1}{\sqrt{|\MM|}}\partial_b\Big(f \partial_a( \sqrt{|\MM|}\A^{ab}) \Big)
= 
\frac{1}{\sqrt{|\MM|}}(\partial_bf)\Big( \partial_a( \sqrt{|\MM|}\A^{ab}) \Big) \\
&=
\frac{1}{\sqrt{|\MM|}}\partial_a \Big(  \sqrt{|G|}\A^{ab}(\partial_b f) \Big) = -
\frac{1}{\sqrt{|\MM|}}\partial_a \Big(  \sqrt{|\MM|} X_f^a\Big)=- \nabla \cdot  X_{f}^\A\,,
\end{align*}
so from \eqref{eq:X_def} applied to $\mu_\M=\vol$ and $f= p_\infty$, we see that 
$$\nabla \cdot \A = -X_\A^{\vol} =- \curl_\vol(\A).$$

 \subsection{Derivations for the Reversibility Section} \label{sec:proofs-reversibility-sec}
 
 Recall that $X^{P}_{\bi}(f) \defn \div_P(X^\bi_f)$ for any bracket $\bi$. Then
 \begin{corollary} The generator of a $P$-preserving expressed in the form of \eqref{eq:intro-A-diff} can be written as 
\begin{equation*}
\L f = \underbrace{ X^{P}_\A(f) +P^{\sharp}(\gamma)(f)}_{L^2(P)\text{-antisymmetric}} +  \underbrace{\half X^{P}_\SS(f)}_{L^2(P)\text{-symmetric}}.
\end{equation*}
Moreover, $\half X^{P}_\SS$ is  symmetric in $L^2(P)$, while $X^{P}_\A$ and $P^{\sharp}(\gamma)$  are both  antisymmetric  in $L^2(P)$, 
$$ \metric{X^P_\A f}{h}_P= -\metric{f}{X^P_\A h}_P , \qquad \metric{X^P_\SS f}{h}_P= \metric{f}{X^P_\SS h}_P,  $$
where $\metric{\cdot}{\cdot}_P$ denotes the $L^2(P)$ pseudo-inner product, $\metric{f}{h}_P \defn \int fh \,\dd P$.
Hence, $\L$ is symmetric  if and only if  $X_\A^P +P^{\sharp}(\gamma)=0$.
In general, the generator of \eqref{general-SDE}  satisfies  $\L = \half X^P_\SS$ if and only if the Fokker-Planck current of $P$  vanishes,   in which case, we say that $\L$ satisfies the \Definition{detailed balance condition}, and the diffusion is \Definition{reversible}.
Finally, $\half X^{P}_\SS$ is non-positive, i.e., 
$$\metric{\half X^{P}_\SS(f)}{f}_P \leq 0$$
for all $f \in C_c^{\infty}(\M)$.
\end{corollary}
\begin{proof}
The formula for the generator follows from \eqref{eq:dissipative-form-A-diffusion} and
 $\div_{P}(X^{\SS}_f)=  \div_P(Y_i)Y_i(f) + Y_iY_if$, as we prove using a local argument: given local coordinates $(z^i)$, writing $P = p_\infty \dd z $ we have
 \begin{align*}
 \div_{P}(X^{\SS}_f) = \frac{1}{p_\infty} \partial_r(p_\infty \SS^{jr} \partial_j f ) =
 \frac{1}{p_\infty} \partial_r(p_\infty Y^j_i Y^r_i \partial_j f )
 = \div_P(Y_i)Y_i(f)+ \partial_r( Y^j_i \partial_j f ) Y^r_i,
 \end{align*}
 which, combined with the fact that for any vector field $Y$,  $YY(f)=Y^i\partial_i(Y^j\partial_jf)=Y^i Y^j\partial_i \partial_j f +Y^i \partial_i Y^j \partial_j f$, yields the result.
From the proof of proposition \refsec{sec:proof-adjoint},
 we know
the (formal) adjoint of $X_P^\A$ is $f \mapsto -\div_P(fX^\A_P)$, and $-\div_P(fX^\A_P)= -X^\A_P(f)-f\div_P(X^\A_P)=-X^\A_P(f)$.
Indeed, this only requires the fact that the dynamics $X^\A_P$ preserves  $P$,
and thus still holds when we include the topological obstruction contribution, which has itself vanishing $\div_P$.
From \eqref{Fokker-general} we also know that the adjoint of $X^{\SS}_P$ is
$f \mapsto - \div_P \left( f\div_P(Y_i)Y_i \right)+  \div_{P} \big( \div_{P}(fY_i)Y_i \big)=  \div_{P} \big( Y_i(f)Y_i)=
\div_{P} \big( X^\A_f)= X^\SS_P(f) $.
\end{proof}

We denote the pushforward with respect to a diffeomorphism $\mathcal{R}$ on tensor fields by $\mathcal{R}_{*}$.  Then

\begin{corollary} 
 Let $\mathcal{R}$ be a target-preserving diffeomorphism,  which is an $\A$-antimorphism  and a $\SS$-morphism, that is
$$\mathcal{R}_*\A=-\A , \qquad \mathcal{R}_*\SS= \SS.$$
 Then, the generator of the $\A$-diffusion \eqref{eq:generator-sym-anti} is \Definition{reversible up to $\mathcal{R}$}, that is we have
$$ \metric{f}{\L h}_P = \metric{\L \mathcal{R}^*f}{\mathcal{R}^*h}_P, \quad \forall f,h \in C_c^{\infty}(\M).$$
\end{corollary}
\begin{proof}
First note that for any $\bi$-Hamiltonian vector field, since
$\mathcal{R}^*P = P$, then, using proposition 6.3.5 \cite{abraham2012manifolds},
$\mathcal{R}^* \L_{X_f}P= \L_{ (\mathcal{R}^{-1})_*X_f}\mathcal{R}^*P=
\L_{ (\mathcal{R}^{-1})_*X_f}P,   $
 so
$  \div_{P}(X_f) \circ \mathcal{R}= \div_{P}((\mathcal{R}^{-1})_*X_f)$.
Now consider  the antisymmetric part, $\L = X_P^\A$. Then
$\metric{f}{X_P^\A h}_P=
-\metric{X_P^\A f}{ h}_P=
-\metric{\mathcal{R}^*X_P^\A f}{ \mathcal{R}^*h}_P$,
where we have used $\mathcal{R}_{\sharp}P = P$ in the last equality.
Moreover $\mathcal{R}^*X_P^\A f = -X_P^\A \mathcal{R}^*f$ since  $\mathcal{R}$ is an $\A$-antimorphism,
$(\mathcal{R}^{-1})_*X_f^\A = -X_{f\circ \mathcal{R}}^\A  $.
The proof of the symmetric part is analogous, except we have
$\mathcal{R}^*X_P^\SS f = X_P^\SS \mathcal{R}^*f$.
Note that if $\mathcal{R}_* Y_i = \pm Y_i$ then $\mathcal{R}$ is a $\SS$-morphism, since $(\mathcal{R}^{-1})_*X_f^\SS =(\mathcal{R}^{-1})_*(Y_i(f)Y_i)
=Y_i(f) \circ \mathcal{R} (\mathcal{R}^{-1})_*Y_i
= \left(\mathcal{R}_* (\mathcal{R}^{-1})_* Y_i \right)(f) \circ \mathcal{R} (\mathcal{R}^{-1})_*Y_i=
\left((\mathcal{R}^{-1})_* Y_i \right) (\mathcal{R}^*f) (\mathcal{R}^{-1})_* Y_i = Y_i(\mathcal{R}^*f) Y_i = X^\SS_{\mathcal{R}^*f} $.
\end{proof}

\subsection{Non-Degenerate and Overdamped Systems}\label{sec:non-degenerate-diffusions}

\begin{theorem}
Any $P \propto p_\infty \vol$ diffusion generated by \eqref{eq:generator-non-degenerate} takes, at least up to topological obstructions, the form \eqref{eq-non-degenerate-general} for some $\A \in \mathfrak X^2(\M)$ ($\nabla$ is the Riemannian gradient) 
\begin{equation*}
\dd Z_t =  \half \nabla \log p_\infty(Z_t)\dd t - X_{\log p_\infty}^\A(Z_t)\dd t - \nabla \cdot \A(Z_t)\dd t + \dd B_t.
\end{equation*} 
\end{theorem}
\begin{proof}
Since $\Delta$ is symmetric in $L^2(\M,\vol)$, 
 $$\L^*p_{\infty}=  -\nabla \cdot (p_\infty X) + \half \Delta p_\infty = \nabla \cdot( -p_\infty X + \half \nabla p_\infty),$$
Thus, following the proof in \refsec{sec:Riemannian-proof}
$$ -p_\infty X + \half \nabla p_\infty= \nabla \cdot (\A  p_{\infty}) = p_{\infty} \nabla \cdot \A  + p_{\infty} X_{\log  p_{\infty} }^\A,$$
at least locally (we refer to \cite{ikeda2014stochastic} for an analysis of non-degenrate measure-preserving diffusions in  the case of compact oriented manifolds),
and thus the drift   has the form
$$ X =\half \nabla \log p_\infty -  \nabla \cdot \A  - X_{\log  p_{\infty} }^\A,  $$
has claimed.
\end{proof}

\subsection{Rate of Change of Functionals along Measure-preserving Diffusions}\label{sec:rate of change appendix}

\begin{proposition} Let $F$ be a  functional on the space of volume measures, and suppose $\frac{\delta F}{\delta Q}\in C_c^{\infty}(\M)$ (or that more generally Stokes' theorem holds).
 The rate of change of $F$ along the $P$-preserving diffusion 
is given by
$$ \frac{\dd }{ \dd t} F(\mu_t) = \left\{\log \frac{\dd P}{\dd \mu_t}, \frac{\delta F}{\delta \mu_t} \right\}_{\int_{\mathcal T}}(\mu_t) +  \metric{ \frac{\dd \mu_t}{\dd P} }{ \mu^{\sharp}_t(\gamma) \left[\frac{\delta F}{\delta \mu_t}\right]  }_{\mu_t}$$
where $\mathcal T$ is the thermodynamic bracket 
$ \SS/\sqrt 2  -\A$, and $\gamma$ the topological obstruction.
\end{proposition}

\begin{proof} Differentiating we have

$$ \frac{\dd }{ \dd t} F(\mu_t) = \int \frac{\delta F}{\delta \mu_t} \frac{\partial \mu_t}{\partial t} = 
\int  \frac{\delta F}{\delta \mu_t} \div_{\mu_t} \left(   X^{\A  }_{\log \frac{\dd P}{\dd \mu_t}}-\half X^{\SS  }_{\log \frac{\dd P}{\dd \mu_t}}\right)\mu_t.$$
If $\frac{\delta F}{\delta \mu_t}$ is compactly supported, or more generally provided Stokes theorem holds,
we have 
$$ \frac{\dd }{ \dd t} F(\mu_t) = -
\int  \left( X^{\A  }_{\log \frac{\dd P}{\dd \mu_t}}\left( \frac{\delta F}{\delta \mu_t}  \right)-\half X^{\SS  }_{\log \frac{\dd P}{\dd \mu_t}} \left( \frac{\delta F}{\delta \mu_t}  \right)\right)\mu_t.$$
Hence
$$ 
\frac{\dd }{ \dd t} F(\mu_t) = - \left\{\log \frac{\dd P}{\dd \mu_t}, \frac{\delta F}{\delta \mu_t} \right\}_{\int_{\A}}(\mu_t) + \half \left\{\log \frac{\dd P}{\dd \mu_t} ,\frac{\delta F}{\delta \mu_t} \right\}_{\int_\SS} (\mu_t) ,$$
which can be expressed as
$$ \frac{\dd }{ \dd t} F(\mu_t) = \left\{\log \frac{\dd P}{\dd \mu_t}, \frac{\delta F}{\delta \mu_t} \right\}_{\int_{\mathcal T}}(\mu_t) $$
via the thermodynamic bracket  $\mathcal T= \SS/\sqrt 2  -\A$ of the diffusion.

For the topological obstruction contribution,
note  that since $\mu_t^{\sharp}(\gamma)$ is $\mu_t$ preserving,
$$ -\int  \frac{\delta F}{\delta \mu_t} 
\mu_t^{\sharp}(\gamma)\left(\frac{\dd \mu_t}{\dd P} \right) \mu_t =
\int \mu_t^{\sharp}(\gamma) \left( \frac{\delta F}{\delta \mu_t}\right) \frac{\dd \mu_t}{\dd P} \mu_t .$$

Note that when $F \defn \mathrm{KL}(\cdot \| P)$ we have 
$$ \frac{\delta \mathrm{KL}(\cdot \| P) }{\delta \mu_t}= \log \frac{\dd \mu_t}{\dd P}= - \log \frac{\dd P}{\dd \mu_t},$$
so we recover the formula for the rate of change of KL. 
\end{proof}

\subsection{Underdamped Langevin on Manifolds}\label{sec:underdamped Langevin manifolds}

\begin{theorem}
Suppose $(\A,\mu_{T\M})$ is a Langevin pair. 
If we choose the noise fields to be the vertical fields $Y_i\defn \mathcal V\circ X_i \circ \pi$ for $X_i \in \mathfrak X(\M)$, then the $\A$-diffusion generated by $e^{-H}\mu_{T\M}$, with $H \defn \pi^*V + \half \| \cdot \|^2$, is
\begin{equation*}
\diff (q_t,v_t) = \underbrace{X_{H}^\A(q_t,v_t)\diff t}_{\A-\text{Hamiltonian dynamic}}- \underbrace{\frac{\beta}{2} \metric{X_i(q_t)}{v_t}_{q_t}\mathcal V ( X_i(q_t)) \diff t}_\text{vertical kinetic Dissipation} +\underbrace{\mathcal V ( X_i(q_t)) \circ \diff W_t}_\text{vertical randomness},
\end{equation*}
or in tangent-lifted coordinates
 \begin{equation*} 
\diff (q_t,v_t) = \Big( X_{H}^\A(q_t,v_t)-\frac{\beta}{2}M(q_t)v_t \Big) \diff t +\sigma(q_t) \circ \diff W_t,
\end{equation*} 
where $M(q)v=M(q)_{rj}v^j \partial_{v^r} \defn (\sigma \sigma^{\top}\MM(q))_{rj}v^j\partial_{v^r}$ and $\sigma \defn \sigma_{ji}\partial_{v^j}  \defn (X_i)^j\partial_{v^j}$,
and
 $\MM_{ij}(q) = \metric{\partial_{x^i}}{\partial_{x^j} }_q$.
In particular $(\Pi,\omega^{n}_\flat)$ is a Langevin pair, where $\Pi$ is the Poisson 2-vector field associated to $\omega_\flat$.

 \end{theorem}
 \begin{proof}
Recall the definition of the vertical lift $ \mathcal V: v_q \mapsto \ver_{v_q}v_q \in T\F $, with $\ver_{v_q}v_q : f \mapsto \frac{\diff }{\diff t } f(q, v_q+tv_q)|_{t=0}$ for any $f\in C^1(\F)$.
For example if $\F = \R^n \times \R^\ell$ is a vector bundle of $\M = \R^n$,  this is just the directional derivative of $f$ at $(q,v_q)$ in the direction $(0,v_q)$, $(0,v) \cdot \nabla f(q,v) = v^k \partial_{v^k} f (q,v)$, or $\mathcal V = v^k \partial_{v^k}$, where $(x^s,v^j)$ are coordinates on $\F$.
If $X \in \mathfrak X(\M)$ is a vector field,  with local expansion $X=X^k(x) \partial_{x^k}$, then its composition with the canonical vector field $\mathcal V \circ X \circ \pi\in \mathfrak X(T\F)$
 is
$\mathcal V \circ X \circ \pi(x,v)=X^k(x) \partial_{v^k}$. 
Then  if $\MM$ is a vector bundle Riemannian metric on $\F$, and $T$ is the associated kinetic energy we find
\begin{align*}
Y_i(H)(x,v)&=
 X^k_i(x) \partial_{v^k}\big( V(x)+T(x,v))= X^k_i(x) \partial_{v^k}\big( \half \MM _{rj}(x) v^rv^j)
 \\
 &= X^k_i(x)\MM _{jk}(x)v^j = \metric{Y_i(x)}{v}_x.
\end{align*} 

Now consider $\F = T\M$. The above derivation shows we  can also write
$$ Y_i(H)(x,v) = \metric{X_i(x)}{v}_x,$$
from which \eqref{global-Vertical-SDE} follows.
Moreover, setting $\sigma_{ji} \defn X^j_i$, so $Y_i(x,v) = \sigma_{ji}(x)\partial_{v^j}$ where $v^j$ are the tangent-lifted coordinates \cite{Abraham:2008}.
 We have 
\begin{align*}
Y_i(H)Y_i(x,v)
 &= \metric{X_i(x)}{v}_x \ver_{X_i(x)}(X_i(x))  = X^k_i(x)\MM _{jk}(x)v^j X^r_i(x)\partial_{v^r}\\
& = \sigma_{ki}(x)\MM _{jk}(x)v^j \sigma_{ri}(x)\partial_{v^r} = \big( \sigma \sigma^{\top}(x) \big)_{rk}\MM _{kj}(x)v^j\partial_{v^r},
\end{align*}
and  the local expression \eqref{Vertical-SDE} follows.
Moreover  the symplectic measure is indeed horizontal (this fact may be traced back to the fact that the Liouville 1-form is horizontal). 
 Indeed, 
 in tangent-lifted coordinates $\omega_\flat^n = |\MM| \diff x  \diff v$, so 
 locally the divergence of $Z \in \mathfrak X(T\M)$
is $\div_{\omega_\flat}(Z)= \frac1{|\MM| } \partial_{x^i}\big(|\MM| Z^i \big)+ \partial_{v^j} \overline Z^j$, and from the previous local expressions we see $\div_{\omega_\flat}(Y_i)=0$, and thus $\omega^n_\flat$ is horizontal.

Finally we mention that if the noise vector fields are chosen to be $\Pi$-Hamiltonian vector fields associated to ``noise'' Hamiltonians  $U_i: \M  \to \R$, then $Y_i \defn X_{U_i\circ \pi}^\Pi =-\mathcal V \circ \nabla U_i \circ \pi$, i.e.,  they are the vertical lift of the Riemannian gradients,
and
$Y_i(H) = \omega_\flat\big(X_H^\Pi, X_{U_i\circ \pi} ^\Pi  \big) = -X_H^\Pi(U_i\circ \pi)
= - \diff U_i \circ \tang \pi (X_H^\Pi)=
- \diff U_i \circ \tang \pi (X_T^\Pi)$
since $X_V$ is vertical.

 \end{proof}

\section{Unpublished Result}\label{sec:unpublished-results}

The following results were proven in the thesis of one of the authors \cite{barp2020bracket}, and 
 are being submitted as part of an article discussing the intrinsic geometry of smooth measures and its relations to various fields of mathematics.
When this latter paper will be available online,
this section will be erased and the mentions of it in the main  article will be replaced by citations, but in the mean time, for completeness, we include the characterisation \refsec{thrm:Dyn-is-loc-curl} of measure-preserving dynamical systems and its proof that we will
use in the main article.

Denote by $\Omega_\Or^k(\M)$  the space of twisted differential $k$-forms, that is differential $k$-forms taking value in the orientation bundle.
In particular the smooth positive measure $P$ can be identified with  a twisted  form of top rank.
For any  $X\in  \mathfrak X^k(\M)$, using $( \mathfrak X^\ell (\M))^* \cong \Omega^\ell(\M)$,
we define the right interior product by  $i_XP (A) \defn \metric{P}{ X \wedge A}_*$ for any $A\in  \mathfrak X^{n-k}(\M) $.
This induces the $C^{\infty}(\M)$-linear musical isomorphism $P ^\flat :  \mathfrak X^k(\M) \rightarrow \Omega_\Or^{n-k}(\M)$ by $P ^\flat (X) \defn i_{X}P $, and we denoted its inverse by $P ^\sharp$, $P^{\sharp} \circ P^{\flat} =\Id$ \cite[Sec. 2.5]{dufour2006poisson}.

Note that since $P^{\flat}(X)$ is twisted,
it is a form that takes value in the orientation line bundle.
Since the orientation bundle is flat,
we can find  transition functions that   are locally constant (in fact these are given by the sign of the Jacobian of the transition functions of $\M$).
Hence the exterior derivative $\dd$ on differential forms extend to an operator on twisted forms, which we will  use in the definition of $\curl_P$ below. Moreover $\dd $ generates a
 canonical twisted de Rham complex  by extending $\dd $ using any trivialisation of the orientation line bundle induced by a trivialisation of $\M$, as explained in
section 7 \cite{bott2013differential} (on orientable manifolds this reduces do the standard de Rham complex).
Using this (extended) exterior derivative, we define the $P$-rotationnel as
$$ \curl_P = P^{\sharp} \circ \dd \circ P^{\flat},$$
which satisfies $\curl_P \circ \curl_P=0$ since $\dd \circ \dd =0$.
In particular when applied to vector fields the $P$-rotationnel acts as the divergence operator $\curl_P= \div_P:\mathfrak X(\M) \to C^{\infty}(\M)$, which follows from
$$ \L_XP = \dd i_X P = P^{\flat} \curl_P(X) = \curl_P(X) P, $$
together with the definition of $\div_P(X)$ as the function satisfying $\div_{P }(X) P  = \mathcal L_X P$.
If $f$ is a function, observe that $(fP)^{\flat} = f P^{\flat}$, and so if $f$ is non-vanishing, $(fP)^{\sharp} = P^{\sharp} \circ \frac1 f$.
Hence
\begin{equation}
 \curl_{fP} = P^{\sharp} \circ \frac1 f \circ \dd \circ f \circ P^{\flat} \defn
P^{\sharp} \circ \dd_f \circ P^{\flat}, \qquad \text{or} \qquad \curl_{Q} = P^{\sharp} \circ \dd_{\frac{\dd Q}{\dd P}} \circ P^{\flat},
\end{equation}
where $\dd_f \defn \dd + \dd \log | f| \wedge$ is the  distorted  de Rham derivative. In particular, the $P$-rotationnel does not depend on the normalisation constant of $P$, an important requirement in many statistical applications, where the target distribution or statistical model is only known up to normalisation.
Importantly, we have the following key result showing the homology groups defined by the boundary operator $\curl_P$ (since $\curl_P \circ \curl_P =0$) are isomorphic to the twisted de Rham cohomology groups. As usual we denote by $[\cdot]$ the equivalence classes.
\begin{theorem} \label{thrm:Dyn-is-loc-curl}
The isomorphism $P^{\flat}$ descends to an isomorphism between the homology groups $\mathcal H_P^\ell(\M)$ of $\curl_P$ and the twisted de Rham cohomology groups $\H^{n-\ell}_{dR}(\M)$. Hence
$$ \Dyn(P )  \cong \curl_P \left( \mathfrak X^2(\M)\right) \oplus P^{\sharp} \left( \H^{n-1}_{dR} \left(\M \right) \right),  $$
and any $P$-preserving dynamics will be globally the $P$-rotationnel of some $\A \in \mathfrak X^2(\M)$ iff the $(n-1)$ de twisted Rham cohomology is trivial.
Moreover, if $U \subset \M$ is an open subset, then $\curl_P |_U = \curl_{P |_U}$.
Hence the set of $P $-preserving dynamics is precisely the set of locally curl vector fields $$ \Dyn(P )  = \text{Curl}_{loc}(P )$$
where $ \text{Curl}_{loc}(P ) =\{ X \in \mathfrak X(\M) : \forall q, \text{there is a neighbourhood } U  \text{ and } \A \in \mathfrak X^2(U) { s.t., }  X = \curl_{P|_U}(\A)   \}$.
\end{theorem}
\begin{proof}
The $P $-derivative 
$\curl_P  \defn P ^{\sharp} \circ \diff \circ P ^{\flat} : \mathfrak X^{k}(\M)  \to \mathfrak X^{k-1}(\M)$ is a vector space homomorphism, satisfying $\curl_P  \circ \curl_ P  =0 $, that it is a boundary operator on the chain complex of $k$-multi-vector fields, and thus 
$\text{Im} \left( \curl_P  : \mathfrak X^{\ell+1}(\M) \to \mathfrak X^{\ell}(\M) \right)$ is a linear subspace of 
$ \ker \left( \curl_P : \mathfrak X^{\ell}(\M) \to \mathfrak X^{\ell-1}(\M) \right)$. 
We can then define the $\ell^{th}$ Holomology group
$$ \mathcal H_P^{\ell} (\M) \defn \frac{\ker \left( \curl_P : \mathfrak X^{\ell}(\M) \to \mathfrak X^{\ell-1}(\M) \right)}{\text{Im} \left( \curl_P  : \mathfrak X^{\ell+1}(\M) \to \mathfrak X^{\ell}(\M) \right)},$$
and in particular the first one provides information on $P$-preserving vector fields
 (using $\curl_P  (X) = \div_P  (X)$)
 $$ \mathcal H_P^{1} (\M) \defn \frac{\ker \left( \div_P : \mathfrak X(\M) \to C^{\infty}(\M)\right)}{\text{Im} \left( \curl_P  : \mathfrak X^{2}(\M) \to  \mathfrak X(\M) \right)}.$$
 Notice that the map $[\cdot ] \circ P^{\flat} :  \mathfrak X^\ell(\M)  \to \H^{n-\ell}_{dR}(\M)$ descends to a map $ \mathcal H_P^{\ell} (\M)  \to \H^{n-\ell}_{dR}(\M)$, since 
 $[\cdot ] \circ P^{\flat}(X+\curl_P(Y))=
 [ P^{\flat}(X)+\dd i_YP]= [P^{\flat}(X)]$.
 The map is surjective since $P^{\flat}$ is, and injective since 
 $ [P^{\flat}(X)] =  [P^{\flat}(Y)]$ $\implies$ $[P^{\flat}(X-Y)]=0 \implies
 P^{\flat}(X-Y) = \dd \alpha \implies X-Y= \curl_P(Z) \implies [X]=[Y]$.  
In particular, $\H^{n-1}_{dR}(\M)$ is trivial iff  $\mathcal H_P^{1} (\M)$ is, in which case every divergence-free vector field is the $\curl_P$ of a bi-vector field $\A$.
 
In general, we can still use Poincaré lemma (or Volterra theorem, as it was proved by Vito Volterra \cite{volterra1889generalisation}) and the properties of $\curl_P$ to show that   any $P $-preserving vector field is  locally a curl vector field.  
Denoting the inclusion by $\iota _U:U \hookrightarrow \M$, we have 
$   \iota _U^*(P ^{\flat}(Y)) = \iota _U^*(i_Y P ) = i_{Y_U} \iota _U^*P =  i_{Y_U} P_{U} = P_{U}^{\flat}(Y_U) $, where $P_{U},Y_U$ denote their restriction to $U$, and we have used proposition 7.4.10 \cite{abraham2012manifolds};
hence $ \iota _U^*\circ P ^{\flat} =  P_{U}^{\flat} \circ |_U$. 
Setting $Y = P^{\sharp}(\alpha)$
this yields
$P^{\sharp}(\alpha) |_U =P_U^{\sharp}(\iota^*_U \alpha)$ for any twisted form $\alpha$.
Hence 
$$|_U \circ \curl_P  \defn |_U \circ P^{\sharp} \circ \dd \circ P^{\flat} =
  P^{\sharp}_U \circ\iota_U^* \circ \dd \circ P^{\flat}=
  P^{\sharp}_U  \circ \dd \circ\iota_U^* \circ P^{\flat}=
   P^{\sharp}_U  \circ \dd \circ P^{\flat}_U 
 =\curl_{P_U} \circ |_U.$$
Thus,  $\curl_P  (Y)=0$
iff $P ^{\sharp} \circ \diff \circ P ^{\flat} (Y)=0$ iff $ \diff \circ P ^{\flat}(Y)=0$ (since $P ^{\sharp}$ is a linear isomorphism). 
By Poincar\'e Lemma, this holds iff around any point there is an open  neighbourhood $U$ over which $ P ^{\flat}(Y) |_{U} = \diff \alpha$ for some twisted ($\dim(\M)-2)$-form $\alpha$ on $U$.
Then 
$  \diff \alpha= \iota _U^*(P ^{\flat}(Y)) =  P_{U}^{\flat}(Y_U) $,  
 and  $P_{U}^{\flat}(Y|_U) = \diff \alpha$ iff  
$Y_U = P ^{\sharp}_U \diff \alpha = \curl_{P_U}  (\A)$ where $\A \defn P ^{\sharp}_U (\alpha) $ is a 2-vector field on $U$.

Finally we also mention that
when $P=\dd x$ is the Lebesgue measure on Euclidean space, $\Dyn(\dd x) = \curl_{\dd x} \left(\mathfrak X^2(\M) \right)$ was essentially already proved by Vito Volterra \cite{volterra1889generalisation}, 
that the statement $\Dyn(P )  \cong \curl_P \left( \mathfrak X^2(\M)\right) \oplus P^{\sharp} \left( \H^{n-1}_{dR} \left(\M \right) \right)$ 
appears in implicit form (essentially written as
$P^{\flat}\left(\Dyn(P )\right) = \dd \Omega^{n-2}(\M) \oplus H^{n-1}(\M)$) in \cite[Thm. 6]{mclachlan2002splitting} under the assumption that $\M$ is orientable,
and that 
 by Poincaré duality if $\M$ has a finite good cover (in which case the cohomology groups are finite dimensional) we may alternatively work with the first compactly supported de Rham  cohomology group (proposition 5.3.1 and theorem 7.8 \cite{bott2013differential}).
\end{proof}
Note that  by Poincaré duality if $\M$ has a finite good cover  we may alternatively work with the first compactly supported de Rham  cohomology group.
The de Rham cohomology groups may be very large, though they must be finite dimensional when $\M$ is compact.
Moreover, in that case,  they are isomorphic to the vector spaces of harmonic forms, and the following results follows:
\begin{corollary} If $\M$ is compact and orientable, then
$$
\Dyn(P )  \cong \curl_P \mathfrak X^2(\M) \oplus P^{\sharp} \left( \mathcal H^{n-1} \left(\M \right) \right),
$$
where $\mathcal H^{n-1}$ is the space of harmonic $n-1$-forms associated to an arbitrary Riemannian metric. In other words, any $P$-preserving vector field on a compact orientable manifold has the form
$X = \curl_P(\A) + P^{\sharp}(\gamma)$.
\end{corollary}

\vskip 0.2in
\bibliographystyle{splncs04}
\bibliography{ReferencesTHMC}

\end{document}